\documentclass[secnum,leqno,10pt]{amsart}
\usepackage{amssymb, amsmath, amsthm}
\usepackage{color}
\usepackage{framed}
 {\endMakeFramed}
\usepackage{mathrsfs}
\usepackage{graphics}
\usepackage{geometry}
\usepackage{indentfirst}
\usepackage[dvips]{graphicx}
\usepackage{eepic}
\usepackage{epic}
\usepackage{amscd}
\newtheorem{thm}{Theorem}[section]
\newtheorem{cor}[thm]{Corollary}

\newtheorem{lmm}[thm]{Lemma}

\newtheorem{prp}[thm]{Proposition}

\newtheorem{dfn}[thm]{Definition}

\theoremstyle{remark}
\newtheorem*{rmk*}{Remark}
\newtheorem*{ack}{\rm{\textbf{Acknowledgement}}}

\makeatletter

\@addtoreset{equation}{section}
\makeatother
\title[]{Traces of CM values and cycle integrals of polyharmonic Maass forms}
\author[]{Toshiki Matsusaka}
\date{}
\address{Graduate School of Mathematics, Kyushu University, Motooka 744, Nishi-ku Fukuoka 819-0395, Japan,}
\email{toshikimatsusaka@gmail.com}


\begin{document}
\maketitle

\begin{abstract}
Lagarias and Rhoades generalized harmonic Maass forms by considering forms which are annihilated by a number of iterations of the action of the $\xi$-operator. In our previous work, we considered polyharmonic weak Maass forms by allowing the exponential growth at cusps, and constructed a basis of the space of such forms. This paper focuses on the case of half-integral weight. We construct a basis as an analogue of our work, and give arithmetic formulas for the Fourier coefficients in terms of traces of CM values and cycle integrals of polyharmonic weak Maass forms. These results put the known results into a common framework.
\end{abstract}

% -----------------------------------------------------------

\section{Introduction}\label{s1}
Traces of CM values have been the subject of extensive research. For a negative integer $d$, we denote by $\mathcal{Q}_d$ the set of positive definite integral binary quadratic forms of discriminant $d$. Under the usual right action of $\mathrm{SL}_2(\mathbb{Z})$, it has finitely many classes $\mathcal{Q}_d/\mathrm{SL}_2(\mathbb{Z})$, whose order is so-called the class number. For each $Q \in \mathcal{Q}_d$, the order of the stabilizer $w_Q^{} := |\mathrm{PSL}_2(\mathbb{Z})_Q|$ is equal to $3, 2,$ or $1$ according as $Q$ is $\mathrm{SL}_2(\mathbb{Z})$-equivalent to $X^2+XY+Y^2$, $X^2+Y^2$ up to a constant multiple, or otherwise, respectively. Under these notations, the Kronecker-Hurwitz class number $H(|d|)$ is defined by
\begin{align*}
	H(|d|) := \sum_{Q \in \mathcal{Q}_d/\mathrm{SL}_2(\mathbb{Z})}\frac{1}{w_Q^{}} \quad \text{for }d<0,
\end{align*}
and $H(0) = -1/12$ additionally. In 1975, Zagier \cite{Z2} showed the generating function $\sum_{d \leq 0}H(|d|) q^{-d}$ is the holomorphic part of a certain harmonic Maass form of weight $3/2$ on $\Gamma_0(4)$. Here we put $q := e^{2\pi iz}$ for $z \in \mathfrak{H} := \{z=x+iy\ |\ y>0\}$. (There is a number of studies about this generating function. For example, \cite{BF, HZ, MM} and so on). He also considered the case replacing $1$, the numerator in the definition of $H(|d|)$, with the elliptic modular $j$-function, that is,
\begin{align}\label{tracej}
	\mathrm{Tr}_d(j-744) := \sum_{Q \in \mathcal{Q}_d/\mathrm{SL}_2(\mathbb{Z})}\frac{j(\alpha_Q)-744}{w_Q^{}} \quad \text{for }d<0,
\end{align}
with additional terms $\mathrm{Tr}_0(j-744) := 2, \mathrm{Tr}_{1}(j-744) := -1$. Here $\alpha_Q$ is the unique root of $Q(z,1)=0$ in the upper half plane $\mathfrak{H}$. Then his 2002 paper \cite{Z} asserts that the generating function $\sum_{d \leq 1}\mathrm{Tr}_d(j-744)q^{-d}$ is a weakly holomorphic modular form of weight $3/2$.\\

A simple question is whether one can give a similar result for another $\mathrm{SL}_2(\mathbb{Z})$-invariant function $f$. For a negative integer $d$, let $\mathrm{Tr}_d(f)$ be the modular trace function replacing $j(\alpha_Q) -744$ with $f(\alpha_Q)$ in (\ref{tracej}). Under this notation, the Kronecker-Hurwitz class number $H(|d|)$ can be written as $H(|d|) = \mathrm{Tr}_d(1)$. For example, we consider the function $f(z) = -\mathrm{log}(y|\eta(z)|^4)$ appearing in Kronecker's first limit formula, where $\eta(z) := q^{1/24}\prod_{n=1}^{\infty}(1-q^n)$ is the Dedekind eta function. This function is not holomorphic, but a sesquiharmonic Maass form of weight $0$, according to Lagarias-Rhoades' recent work \cite{LR}. There are a lot of famous works on the trace $\mathrm{Tr}_d ( -\log(y |\eta(z)|^4))$, for example, Lerch, Chowla-Selberg, Deninger, reviewed in \cite[(40), (64)]{DIT5}. As an analogue of Zagier's results, does the generating function
\begin{align}\label{KLF-gen}
	\sum_{d < 0} \mathrm{Tr}_d (-\log(y|\eta(z)|^4)) q^{-d}
\end{align}
have a suitable modularity? One of our goals is to realize such modularity results for any polyharmonic weak Maass form $f$ of weight $0$.\\

To make this precise, we first recall what polyharmonic Maass forms are. Poyharmonic Maass forms are a generalization of harmonic Maass forms. They are annihilated by a number of iterations of the action of the hyperbolic Laplace operator $\Delta_k$. An example of such forms of depth $3/2$ is studied in Duke-Imamo\={g}lu-T\'{o}th's paper \cite{DIT2}, and generally defined as sesquiharmonic Maass forms (``sesqui" means $3/2$) by Bringmann-Diamantis-Raum \cite{BDR}. More generally, Lagarias-Rhoades \cite{LR} introduced polyharmonic Maass forms to study the roles of the higher Laurent coefficients of the real-analytic Eisenstein series. Furthermore polyharmonic weak Maass forms, which are allowed the exponential growth at cusps, of integral weight were studied in our previous paper \cite{M}. Related works on polyharmonic Maass forms are \cite{AAS, ALR, BFI, JKK4} and so on. In this article, we are also concerned with the case of half-integral weight.\\

Throughout this article, for a fixed $k \in \frac{1}{2}\mathbb{Z}$, let $\Gamma = \mathrm{SL}_2(\mathbb{Z})$ if $k \in \mathbb{Z}$, and $\Gamma = \Gamma_0(4)$ if $k \in \mathbb{Z}+1/2$. Then for any $\gamma = [\begin{smallmatrix}a & b \\c & d \end{smallmatrix}] \in \Gamma$, set
\begin{align*}
	j_k(\gamma,z) := \left\{\begin{array}{ll}
		\sqrt{cz+d} \quad &\text{if } k \in \mathbb{Z},\\
		\bigl(\frac{c}{d}\bigr)\epsilon_d^{-1}\sqrt{cz+d}\quad &\text{if } k \in \mathbb{Z}+1/2,
	\end{array} \right.
\end{align*}
where $\sqrt{z}$ is the principal branch of the holomorphic square root, $\bigl(\frac{\cdot}{\cdot}\bigr)$ is the Kronecker symbol, and $\epsilon_d = 1$ if $d \equiv 1 \pmod{4}$ and $\epsilon_d = i$ if $d \equiv 3 \pmod{4}$. Then we define the weight $k$ slash operator by $(f|_k\gamma)(z) = j_k(\gamma,z)^{-2k}f(\gamma z)$. For a fixed $r \in \frac{1}{2}\mathbb{Z}_{>0}$, a complex-valued smooth function $f$ on the upper half plane $\mathfrak{H}$ satisfying
\begin{enumerate}
	\item $(f|_k \gamma) (z) = f(z)$ for any $\gamma \in \Gamma$, 
	\item $f$ is annihilated by the repeating action of the $\xi$-operator $\xi_k := 2i y^k \overline{\frac{\partial}{\partial \bar{z}}}$, $2r$ times, that is,
		\begin{align*}
			\xi_a \circ \cdots \circ \xi_k \circ \xi_{2-k} \circ \xi_k f(z) =0,\quad \text{($2r$ times)},
		\end{align*}
	where $a = k$ if $2r$ is odd, and $a = 2-k$ if $2r$ is even,
	\item there exists an $\alpha \in \mathbb{R}_{>0}$ such that $f(x+iy) = O(y^{\alpha})$ as $y \to \infty$, uniformly in $x \in \mathbb{R}$,
\end{enumerate}
is called a polyharmonic Maass form of weight $k$ and depth $r$. 

\begin{rmk*}
	This definition of depth may look a little strange. We recall that harmonic Maass forms are annihilated by the action of the hyperbolic Laplace operator
	\[
		\Delta_k := -y^2\biggl(\frac{\partial^2}{\partial x^2} + \frac{\partial^2}{\partial y^2}\biggr) +iky\biggl(\frac{\partial}{\partial x} + i\frac{\partial}{\partial y}\biggr).
	\]
	Then polyharmonic Maass forms should be naturally characterized by $\Delta_k^r f(z) = 0$. The depth originally means the number of iterations of $\Delta_k$. On the other hand, the $\xi$-operator is the ``half-iteration" of the hyperbolic Laplace operator in the sense of $\Delta_k = -\xi_{2-k} \circ \xi_k$. Therefore we define the depth of polyharmonic Maass forms as above.
\end{rmk*}

The space of all such forms is denoted by $H_k^r$. We easily see that $H_k^r$ is a $\mathbb{C}$-vector space. In particular, an element of $H_k^1$ is a harmonic Maass form, and $H_k^{1/2}$ is the space of holomorphic modular forms. Moreover, we call them sesquiharmonic if $r=3/2$. Further, we let $H_k^{r,!}$ be the larger space of polyharmonic weak Maass forms defined by relaxing the third condition to $f(x+iy) = O(e^{\alpha y})$. Moreover in the half-integral weight case, we write $H_k^{r,+,!}$ for the subspace of $f \in H_k^{r,!}$ satisfying Kohnen's plus-condition, that is, the $n$-th Fourier coefficient $a(n,y)$ of $f \in H_k^{r,!}$ in the form $f(z) = \sum_{n \in \mathbb{Z}}a(n,y)e^{2\pi inx}$ vanishes unless $(-1)^{k-1/2}n \equiv 0, 1 (\text{mod }4)$. Then it can be easily checked that if $f \in H_k^{r,+,!}$ has polynomial (exponential) growth at $i \infty$, then it also has at most polynomial (exponential) growth at other cusps $0, 1/2$. Throughout this article, we always assume Kohnen's plus-condition. Hence we drop the notation ``+", and write $H_k^{r,!}$ instead of $H_k^{r,+,!}$ simply.\\

First we construct a basis for this space $H_k^{r,!}$ as an analogue of the author's earlier work \cite{M} by using the Maass-Poincar\'{e} series $P_{k,m}(z,s)$ of half-integral weight $k \in \mathbb{Z} + 1/2$ (defined in Section \ref{s3}). We consider its Laurent expansion
\begin{align*}
	P_{k,m}(z,s) = \left\{\begin{array}{ll}
		\sum_{r \in \mathbb{Z}} F_{k,m,r}(z)\bigl(s+\frac{k}{2}-1\bigr)^r \quad &\text{if } k \leq 1/2, \\
\ \\
		\sum_{r \in \mathbb{Z}} G_{k,m,r}(z)\bigl(s-\frac{k}{2}\bigr)^r \quad &\text{if } k \geq 3/2.
	\end{array} \right.
\end{align*}
Since this series $P_{k,m}(z,s)$ is holomorphic in $\mathrm{Re}(s)>1$, the coefficients satisfy $F_{k,m,r}(z) = G_{k,m,r}(z) = 0$ for $k \neq 1/2, 3/2$ and $r < 0$. As for $k = 1/2$ and $k = 3/2$, Duke-Imamo\={g}lu-T\'{o}th \cite[Lemma 3]{DIT2} and Jeon-Kang-Kim \cite[Section 5]{JKK} considered the analytic continuation to $s = 3/4$, and showed the coefficients also satisfy $F_{k,m,r}(z) =G_{k,m,r}(z) = 0$ if $r< 0$ except for $F_{1/2,n^2,-1}(z) \neq 0$ with $n \geq 0$. For the space $H_k^{1/2,!} = M_k^{+,!}$ of weakly holomorphic modular forms, Duke-Jenkins \cite{DJ} constructed a standard basis $\{f_{k,m}(z)\ |\ m \geq -A_k, (-1)^{\lambda_k-1}m \equiv 0,1(4)\}$. This function $f_{k,m}(z)$ has the Fourier expansion $f_{k,m}(z) = q^{-m} + \sum_{n > A_k}a_k(m,n)q^n$. (For more details, see Section \ref{s5}). More generally, we construct a basis for the space $H_k^{r,!}$ of polyharmonic weak Maass forms of half-integral weight in terms of these Laurent coefficients $F_{k,m,r}(z)$ and $G_{k,m,r}(z)$ as follows.

\begin{thm}\label{Main2}
Let $r \geq 1$ be an integer and $k = \lambda_k + 1/2$ with $\lambda_k \in \mathbb{Z}$. We take the decompositions $H_k^{r,!} = H_k^{r-1/2,!} \oplus \mathcal{H}_k^{r,!}$ and $H_k^{r-1/2,!} = H_k^{r-1,!} \oplus \mathcal{H}_k^{r-1/2,!}$ as $\mathbb{C}$-vector spaces. Then we have $\mathcal{H}_k^{0,!} = \{0\}$ and
\begin{enumerate}
	\item For $k \leq -1/2$,
		\begin{enumerate}
			\item $\{F_{k,m,r-1}(z)\ |\ m > A_k, (-1)^{\lambda_k}m \equiv 0,1(4) \}$ is a basis for $\mathcal{H}_k^{r,!}$.
			\item $\{\tilde{F}_{k,m,r-1}(z)\ |\ m \leq A_k, (-1)^{\lambda_k}m \equiv 0,1(4) \}$ is a basis for $\mathcal{H}_k^{r-1/2,!}$.
		\end{enumerate}
	\item For $k = 1/2$,
		\begin{enumerate}
			\item $\{F_{1/2,m,r-1}(z) - 8\sqrt{m}\delta_{\square}(m)F_{1/2,0,r-1}(z)\ |\ 0 < m \equiv 0,1(4) \}$ is a basis for $\mathcal{H}_{1/2}^{r,!}$.
			\item $\{F_{1/2,m,r-1}(z)\ |\ 0>m \equiv 0,1(4) \} \cup \{F_{1/2,0,r-2}(z)\}$ is a basis for $\mathcal{H}_{1/2}^{r-1/2,!}$.
		\end{enumerate}
	\item For $k = 3/2$,
		\begin{enumerate}
			\item $\{G_{3/2,m,r}(z)\ |\ 0<m \equiv 0,3(4)\}\cup\{G_{3/2,0,r-1}(z)\}$ is a basis for $\mathcal{H}_{3/2}^{r,!}$.
			\item $\{G_{3/2,m,r-1}(z) - \frac{4}{\sqrt{\pi}}\delta_{\square}(-m)G_{3/2,0,r-1}(z)\ |\ 0>m \equiv 0,3(4)\}$ is a basis for $\mathcal{H}_{3/2}^{r-1/2,!}$.
		\end{enumerate}
	\item For $k \geq 5/2$,
		\begin{enumerate}
			\item $\{\tilde{G}_{k,m,r}(z)\ |\ m > A_k , (-1)^{\lambda_k}m \equiv 0, 1 (4)\}$ is a basis for $\mathcal{H}_k^{r,!}$.
			\item $\{ G_{k,m,r-1}(z)\ |\ m \leq A_k, (-1)^{\lambda_k}m \equiv 0, 1 (4)\}$ is a basis for $\mathcal{H}_k^{r-1/2,!}$.
		\end{enumerate}
\end{enumerate}
Here we put
\begin{align*}
	\tilde{F}_{k,m,r-1}(z) &:= F_{k,m,r-1}(z) +\sum_{\substack{A_k < n < 0\\(-1)^{\lambda_k}n \equiv 0,1(4)}}a_k(-m,n)F_{k,n,r-1}(z),\\
	\tilde{G}_{k,m,r}(z) &:= m^{k-1}G_{k,m,r}(z) - \sum_{\substack{0 < n \leq A_k\\(-1)^{\lambda_k}n \equiv 0,1(4)}}a_k(-n,m) n^{k-1}G_{k,n,r}(z),\\
	\delta_{\square}(m) &:= \left\{\begin{array}{ll}
		1 &\text{if } m \text{ is a square number},\\
		0 &\text{otherwise}.
	\end{array} \right.
\end{align*}
\end{thm}

In particular by (\ref{xirec1}) and (\ref{xirec2}) in Section \ref{s4}, we immediately obtain the following corollary.

\begin{cor}
	Let $r \in \frac{1}{2}\mathbb{Z}$. The map $\xi_k: H_k^{r,!} \to H_{2-k}^{r-1/2,!}$ is surjective.
\end{cor}

\begin{rmk*}
The surjectivity of $\xi_k: H_k^{1,!} \to H_{2-k}^{1/2,!}$ was shown in \cite[Theorem 3.7]{BF2}. Recently, Jeon-Kang-Kim \cite{JKK2} constructed a basis for $H_k^{1,!}$ for $k \in \frac{1}{2}\mathbb{Z}$.
\end{rmk*}

Next we investigate the Fourier coefficients of these polyharmonic weak Maass forms to our goal. We redefine the set $\mathcal{Q}_d$ for any $d \in \mathbb{Z}$ with $d \equiv 0, 1 \pmod{4}$ by 
\[
	\mathcal{Q}_d = \{ Q(X,Y) = aX^2+bXY+cY^2\ |\ a,b,c \in \mathbb{Z}, b^2-4ac = d\}.
\]
In addition if $d$ is not a square number, we split $\mathcal{Q}_d = \mathcal{Q}_d^+ \sqcup \mathcal{Q}_d^-$ into the subsets of forms with $a>0$ and $a<0$. For each fundamental discriminant $D$ or $D=1$, an integer $d$ with $dD <0$, and a $\mathrm{SL}_2(\mathbb{Z})$-invariant function $f$, we define the twisted trace by
\[
	\mathrm{Tr}_{d,D}(f) := \sum_{Q \in \mathcal{Q}^+_{dD} /\mathrm{SL}_2(\mathbb{Z})} \frac{\chi_D^{}(Q)}{w_Q^{}}f(\alpha_Q),
\]
where $\chi_D^{}$ is a genus character defined by
\begin{align*}
	\chi_D^{}(Q) := \left\{\begin{array}{ll}
		\bigl(\frac{D}{r}\bigr), \quad &(a,b,c,D)=1\text{ and }(r,D)=1\text{ where }Q\text{ represents }r,\\
		0, \quad &(a,b,c,D) >1
	\end{array} \right.
\end{align*}
on $\mathcal{Q}_{dD}/\mathrm{SL}_2(\mathbb{Z})$. For the details, see \cite[Section I-2]{GKZ}. On the other hand, according to Duke-Imamo\={g}lu-T\'{o}th \cite{DIT2}, for a non-square $dD >0$, we define the twisted trace of cycle integrals by
\[
	\mathrm{Tr}_{d,D}(f) := \frac{1}{2\pi}\sum_{Q \in \mathcal{Q}_{dD}/\mathrm{SL}_2(\mathbb{Z})}\chi_D^{}(Q)\int_{\mathrm{SL}_2(\mathbb{Z})_Q \backslash S_Q}f(z)\frac{\sqrt{dD}dz}{Q(z,1)},
\]
where $S_Q$ is the geodesic characterized by the semi-circle $a|z|^2+b\mathrm{Re}(z)+c=0$ oriented counter-clockwise if $a>0$ and clockwise if $a<0$, and $\mathrm{SL}_2(\mathbb{Z})_Q$ is the stabilizer of $Q$. Since the cycle integral is invariant under changing $Q$ to $-Q$ and it holds that $\chi_D(-Q) = \mathrm{sgn}(D)\chi_D(Q)$, we easily see that $\mathrm{Tr}_{d,D}(f) = 0$ for $d<0$ and $D<0$.\\

To our purpose, we take any polyharmonic weak Maass form $f$ of weight $0$ as an input function. By our earlier work \cite{M}, we know that the space $H_0^{r,!}$ of polyharmonic weak Maass forms of weight $0$ and any depth $r$ is generated by the Laurent coefficients of the Niebur-Poincar\'{e} series $G_m(z,s)$ ($m \in \mathbb{Z}$) at $s=1$. We remark that in our paper \cite{M} we considered the Maass-Poincar\'{e} series of weight 0 instead of the Niebur-Poincar\'{e} series, but they are the same up to the gamma function. The precise definition is given in Section \ref{s3.2}. Thus it is enough to consider the functions $f(z) = F_{0,m,r}(z)$, where we take the Laurent expansion
\[
	G_m(z,s) = \sum_{r=-1}^{\infty} F_{0,m,r}(z)(s-1)^r.
\]
Since the function $G_0(z,s)$ is the usual real analytic Eisenstein series $E(z,s)$, we see that $F_{0,0,-1}(z) = 3/\pi$ and $F_{0,0,0}(z) = -(3/\pi) \log(y|\eta(z)|^4) + C$, where $C = (6/\pi) (\gamma - \log 2 - 6\zeta'(2)/\pi^2)$ and $\gamma$ is the Euler constant. As for $m \neq 0$, it is known that $F_{0,m,-1}(z) = 0$. (See \cite[Section 3.2]{DIT}). Then we obtain the following modularity results.

\begin{thm}\label{Main4}
	For the polyharmonic weak Maass form $F_{0,m,r}(z)$ and a fundamental discriminant $D$ or $D=1$, the generating functions
	\begin{align*}
		&\sum_{d \in \mathbb{Z}}\mathrm{Tr}_{d,D}^{\mathrm{add}}(F_{0,m,r})q^d + \sum_{\substack{0<d \equiv 0, 1(4)\\dD = \square}}d^{-1/2}\mathrm{Tr}_{d,D}^{\mathrm{sq}}(F_{0,m,r})q^d + \sum_{\substack{0 < d \equiv 0,1 (4)\\ dD \neq \square}} d^{-1/2}\mathrm{Tr}_{d,D}(F_{0,m,r})q^d,\\
		&\sum_{d \in \mathbb{Z}}\mathrm{Tr}_{d,D}^{\mathrm{add}}(F_{0,m,r})q^d +\sum_{0 > d \equiv 0,1 (4)} \mathrm{Tr}_{d,D}(F_{0,m,r})q^{-d}
	\end{align*}
	are the holomorphic parts of polyharmonic weak Maass forms of weight $1/2$ and $3/2$, respectively. Up to finitely many exceptions, the values $\mathrm{Tr}_{d,D}^{\text{add}}(F_{0,m,r})$ all vanish. These additional numbers $\mathrm{Tr}_{d,D}^{\mathrm{add}}(F_{0,m,r})$ and $\mathrm{Tr}_{d,D}^{\mathrm{sq}}(F_{0,m,r})$ are the appropriate values defined in Section $\ref{s7}$.
\end{thm}

\begin{rmk*}
	The latter series is identically $0$ for a negative $D<0$.
\end{rmk*}

\begin{rmk*}
	The definition of the ``holomorphic part" of a polyharmonic Maass form depends on the choice of Fourier expansion. In this paper, we use the functions $u_{k,n}^{[j], \pm}(y)$ and define the holomorphic part as Definition \ref{defholpart} in Section \ref{s2}.
\end{rmk*}

By this theorem, we see that the generating function (\ref{KLF-gen}) with suitable additional terms becomes the holomorphic part of a polyharmonic Maass form of weight $3/2$ and depth $2$ on $\Gamma_0(4)$. There is a number of studies on such modularities of trace functions: the case of $f=1, j_m(z) \in H_0^{1/2,!}$ and $dD<0$ is in Zagier \cite{Z2, Z}, $f = 1, j_m(z)$ and $dD>0$ is in Duke-Imamo\={g}lu-T\'{o}th \cite{DIT2}. Here $j_m(z)$ is the unique polynomial in $j(z)$ of the form $q^{-m} + O(q)$. For harmonic weak Maass forms, in more general setting, Alfes-Neumann and Schwagenscheidt \cite{AS1, AS2} considered such generating functions for any $d, D$. Moreover for sesquiharmonic Maass forms $\hat{J}_m$ satisfies $\Delta_0 \hat{J}_m = -j_m - 24\sigma_1(m)$, Jeon-Kang-Kim \cite{JKK4} constructed the generating functions of traces of their cycle integrals. This theorem \ref{Main4} gives a generalization of these results to higher depth.\\

Finally we give one application of our results. As shown in Proposition \ref{Theorem 4.4} in Section \ref{s3}, we have the Fourier expansion of the Maass-Poincar\'{e} series $P_{k,m}(z,s)$. By using (\ref{key2}) the work of Duke-Imamo\={g}lu-T\'{o}th \cite{DIT2}, we see that the holomorphic part of the sesquiharmonic Maass form $F_{1/2,0,0}(z) \in H_{1/2}^{3/2}$ is realized as the generating function of $\mathrm{Tr}_{d,D}(-\log(y|\eta(z)|^4))$ with $D \neq 1$, and also $\mathrm{Tr}_{d,1}(1)$. Precisely we get the following statement.

\begin{cor}\label{MainCor}
	Let $D \neq 1$ be a fundamental discriminant. The $d$-th $(0 < d \neq \square$ and $dD \neq \square)$ Fourier coefficient of the holomorphic part of $F_{1/2,0,0}(z)$ gives the equation
	\begin{align*}
		\mathrm{Tr}_{d,D}(-\mathrm{log}(y|\eta(z)|^4)) &= \sqrt{|D|}L_D(1)\mathrm{Tr}_{d,1}(1),
	\end{align*}
	where $L_D(s) := \sum_{n=1}^{\infty} \big(\frac{D}{n}\big) n^{-s}$ is the Dirichlet $L$-function.
\end{cor}

For $D<0$, the left-hand side is equal to
\[
	\sum_{Q \in \mathcal{Q}^+_{dD} /\mathrm{SL}_2(\mathbb{Z})} -\frac{\chi_D^{}(Q)}{w_Q^{}}\mathrm{log}(\mathrm{Im}(\alpha_Q)|\eta(\alpha_Q)|^4),
\]
and if $d$ is also a fundamental discriminant, the right-hand side is given by
\begin{align}\label{kronecker}
	\sqrt{|D|}L_D(1)\frac{1}{2\pi}\sum_{Q \in \mathcal{Q}_{d} / \mathrm{SL}_2(\mathbb{Z})}\int_{\mathrm{SL}_2(\mathbb{Z})_Q \backslash S_Q}\frac{\sqrt{d}dz}{Q(z,1)} = \sqrt{|D|}L_D(1)\frac{h(d)\mathrm{log}\varepsilon_d}{\pi},
\end{align}
where $h(d)$ is the narrow class number of $\mathbb{Q}(\sqrt{d})$ and $\varepsilon_d$ is the smallest unit $>1$ with positive norm. This equation (\ref{kronecker}) was originally established by Kronecker \cite[(41)]{DIT5}. (As a good reference, Duke-Imamo\={g}lu-T\'{o}th \cite{DIT5} reviewed the history of Kronecker's limit formula). As for a fundamental discriminant $d>0$ and $D>0$ with $dD \neq \square$, we have
\[
	\sum_{Q \in \mathcal{Q}_{dD}/\mathrm{SL}_2(\mathbb{Z})}\chi_D(Q)\int_{\mathrm{SL}_2(\mathbb{Z})_Q \backslash S_Q}-\mathrm{log}(y|\eta(z)|^4)\frac{\sqrt{dD}dz}{Q(z,1)} = 2\sqrt{D}L_D(1)h(d)\mathrm{log}\varepsilon_d.
\]
This formula was written down by Siegel \cite[(65)]{DIT5}. While Kronecker and Siegel proved these equations only for a fundamental discriminant $d$ that is coprime to $D$, Corollary \ref{MainCor} holds more generally. As an example, in the case of $d=8$ and $D=-4$, since $\mathcal{Q}_{-32}^+ / \mathrm{SL}_2(\mathbb{Z}) = \{ X^2 + 8Y^2, 3X^2 + 2XY + 3Y^2, 2X^2 + 4Y^2\}$, we have
\[
	\log \frac{1}{3} \bigg|\frac{\eta\bigl(\frac{-1+2\sqrt{2} i}{3}\bigr)}{\eta(2\sqrt{2}i)}\bigg|^4 = 2 \cdot \frac{\pi}{4} \cdot \frac{\log (3+2\sqrt{2})}{\pi} = \log(1+\sqrt{2}),
\]
that is,
\[
	\bigg|\frac{\eta\bigl(\frac{-1+2\sqrt{2} i}{3}\bigr)}{\eta(2\sqrt{2}i)}\bigg|^4 = 3+3\sqrt{2}.
\]\\

% -----------------------------------------------------------

The paper is organized as follows. First, in Section \ref{s2}, we give basic properties of Whittaker functions. After that, we review some results on the Maass-Poincar\'{e} series and the Niebur-Poincar\'{e} series. Through Section \ref{s4} to Section \ref{s6}, we give a proof of Theorem \ref{Main2}.  In Section \ref{s7} we prove Theorem \ref{Main4} divided into $6$ parts, and give the definitions of the additional terms. Finally we write down a proof of Corollary \ref{MainCor}.

% -----------------------------------------------------------

\begin{ack}
The author would like to thank Masanobu Kaneko, Soon-Yi Kang, Chang Heon Kim, and Markus Schwagenscheidt for their helpful comments. He also thanks the referees for their comments on a preliminary version of this paper. This work is supported by Research Fellow (DC) of Japan Society for the Promotion of Science, Grant-in-Aid for JSPS Fellows 18J20590 and JSPS Overseas Challenge Program for Young Researchers.
\end{ack}

% -----------------------------------------------------------

\section{The Whittaker functions}\label{s2}%section2

In this section, we recall some basic properties of Whittaker functions based on \cite[9.22--9.23]{GR}, \cite[Chapter VII]{MOS}. For two parameters $\mu, \nu \in \mathbb{C}$, Whittaker functions $M_{\mu,\nu}(z)$ and $W_{\mu,\nu}(z)$ are the standard solutions to the Whittaker differential equation
\begin{align}\label{WDE}
	\frac{d^2w}{dz^2} + \biggl(-\frac{1}{4}+\frac{\mu}{z}+\frac{1-4\nu^2}{4z^2}\biggr)w=0.
\end{align}
If these parameters $\mu, \nu$ satisfy $\mathrm{Re}(\nu \pm \mu +1/2) >0$ and $y >0$, then Whittaker functions are represented by
\begin{align*}
	M_{\mu,\nu}(y) &= y^{\nu+\frac{1}{2}}e^{\frac{y}{2}}\frac{\Gamma(1+2\nu)}{\Gamma(\nu+\mu+\frac{1}{2})\Gamma(\nu-\mu+\frac{1}{2})}\int_0^1 t^{\nu+\mu-\frac{1}{2}}(1-t)^{\nu-\mu-\frac{1}{2}}e^{-yt}dt,\\
	W_{\mu,\nu}(y) &= y^{\nu+\frac{1}{2}}e^{\frac{y}{2}}\frac{1}{\Gamma(\nu-\mu+\frac{1}{2})}\int_1^{\infty}t^{\nu+\mu-\frac{1}{2}}(t-1)^{\nu-\mu-\frac{1}{2}}e^{-yt}dt.
\end{align*}
Moreover, we define a modified version of the Whittaker function $\mathcal{M}_{\mu,\nu}^+(z) := W_{-\mu,\nu}(ze^{\pi i})$ according to the paper \cite{ALR}. This function $\mathcal{M}_{\mu,\nu}^+(z)$ is also a solution of (\ref{WDE}), and always linearly independent to $W_{\mu,\nu}(z)$. Thus the Whittaker differential equation (\ref{WDE}) has two linearly independent solutions $W_{\mu,\nu}(z)$ and $\mathcal{M}_{\mu,\nu}^+(z)$.\\

We now explain the Fourier expansion of polyharmonic weak Maass forms. First we consider the integral depth case. Since any $f \in H_k^{r,!}$ satisfies the modular transformation law, we have $f(z+1) = f(z)$, that is, $f(z)$ has the following Fourier expansion
\[
	f(z) = \sum_{n \in \mathbb{Z}}a(n,y)e^{2\pi inx}.
\]
By the second condition for the definition of polyharmonic Maass forms, it holds that $\Delta_k^rf(z) =0$, where $\Delta_k$ is the hyperbolic Laplacian defined by
\[
	\Delta_k := -y^2\biggl(\frac{\partial^2}{\partial x^2} + \frac{\partial^2}{\partial y^2}\biggr) +iky\biggl(\frac{\partial}{\partial x} + i\frac{\partial}{\partial y}\biggr) = -\xi_{2-k} \circ \xi_k.
\]
Then each coefficient $a(n,y)$ satisfies a certain $2r$-order linear differential equation. For $k \neq 1$, Andersen-Lagarias-Rhoades \cite{ALR} gave $2r$ linear independent solutions by
\begin{align*}
	u_{k,n}^{[j],-}(y) &:= y^{-\frac{k}{2}}\frac{\partial^j}{\partial s^j}W_{\mathrm{sgn}(n)\frac{k}{2}, s-\frac{1}{2}}(4\pi |n|y)\bigg|_{s=\frac{k}{2}},\\
	u_{k,n}^{[j],+}(y) &:= y^{-\frac{k}{2}}\frac{\partial^j}{\partial s^j}\mathcal{M}^+_{\mathrm{sgn}(n)\frac{k}{2}, s-\frac{1}{2}}(4\pi |n|y)\bigg|_{s=\frac{k}{2}},
\end{align*}
for $0 \leq j \leq r-1$ if $n \neq 0$, and for $n =0$,
\begin{align*}
	u_{k,0}^{[j],-}(y) &:= \frac{\partial^j}{\partial s^j}y^{1-\frac{k}{2}-s}\bigg|_{s=\frac{k}{2}} = (-1)^j(\mathrm{log}\ y)^j y^{1-k},\\
	u_{k,0}^{[j],+}(y) &:= \frac{\partial^j}{\partial s^j}y^{s-\frac{k}{2}}\bigg|_{s=\frac{k}{2}} = (\mathrm{log}\ y)^j.
\end{align*}
In the special case of $j=0$, we can express these functions by simple functions,
\begin{align}
	\begin{split}\label{useful1}
		u_{k,n}^{[0],-}(y)e^{2\pi inx} &= y^{-\frac{k}{2}}W_{\mathrm{sgn}(n)\frac{k}{2}, \frac{k-1}{2}}(4\pi |n|y)e^{2\pi inx}\\
		&= \left\{\begin{array}{ll}
			(4\pi n)^{\frac{k}{2}}q^n\quad &\text{if } n > 0, \\
			(4\pi |n|)^{\frac{k}{2}}\Gamma(1-k,4\pi |n|y)q^n\quad &\text{if } n<0,
		\end{array} \right.
	\end{split}\\
	\begin{split}\label{useful2}
		u_{k,n}^{[0],+}(y)e^{2\pi inx} &= y^{-\frac{k}{2}}\mathcal{M}^+_{-\frac{k}{2}, \frac{k-1}{2}}(4\pi |n|y)e^{2\pi inx}\\
		&= (4\pi |n| e^{\pi i})^{\frac{k}{2}} q^n,\hspace{105pt} \text{if } n<0,
	\end{split}
\end{align}
where $\Gamma(s,y) := \int_y^{\infty} e^{-t}t^{s-1}dt$ is the incomplete Gamma function. Moreover by the below Lemma \ref{xiop}, the function $u_{k,n}^{[0],+}(y)e^{2\pi inx}$ with $n>0$ is not holomorphic. Hence we have

\begin{prp}\label{FWE}\cite[Section 3]{ALR}
	Let $f(z) \in H_k^{r,!}$ for $k \in \frac{1}{2}\mathbb{Z}$ and $r \in \mathbb{Z}_{>0}$. Then the Fourier-Whittaker expansion of $f(z)$ is given by
\begin{align*}
f(z) = \sum_{n \in \mathbb{Z}} \sum_{j=0}^{r-1} \biggl(c_{n,j}^- u_{k,n}^{[j],-}(y)e^{2\pi i n x} + c_{n,j}^+ u_{k,n}^{[j],+}(y)e^{2\pi i n x} \biggr),
\end{align*}
where $c_{n,j}^{\pm} \in \mathbb{C}$. If $k \in \mathbb{Z} + 1/2$, then it is required to satisfiy Kohnen's plus-condition.
\end{prp}

Finally, we combine Proposition \ref{FWE} with the condition on the behavior at cusps. Corollary A.3 in \cite{ALR} asserts that $u_{k,n}^{[j],+}(y)$ grows exponentially as $y \to \infty$, while $u_{k,n}^{[j],-}(y)$ decays exponentially as $y \to \infty$ for $n \neq 0$. By the growth condition at $i \infty$, the Fourier coefficients $c_{n,j}^+ = 0$ for almost all indices $(n,j)$. If all coefficients $c_{n,j}^+ =0$ for $n \neq 0$, then $f \in H_k^{r}$ strictly. In order to consider the case of half-integral depth, we recall the following lemma.

\begin{lmm}\cite[Lemma 2.2]{M}\label{xiop}
	Under the above notations, we have
	\begin{align*}
		&\quad\xi_k(u_{k,n}^{[j],-}(y)e^{2\pi i nx})\\
		&\quad\quad=\left\{\begin{array}{ll}
			j(1-k)u_{2-k,-n}^{[j-1],-}(y)e^{-2\pi i nx} - j(j-1)u_{2-k,-n}^{[j-2],-}(y)e^{-2\pi inx} \quad \text{if } n > 0, \\
			-u_{2-k,-n}^{[j],-}(y)e^{-2\pi inx} \quad \text{if } n < 0,
		\end{array} \right.\\
		&\quad\xi_k(u_{k,n}^{[j],+}(y)e^{2\pi i nx})\\
		&\quad\quad = \left\{\begin{array}{ll}
			-u_{2-k,-n}^{[j],+}(y)e^{-2\pi inx} \quad \text{if } n > 0,\\
			j(1-k)u_{2-k,-n}^{[j-1],+}(y)e^{-2\pi i nx} - j(j-1)u_{2-k,-n}^{[j-2],+}(y)e^{-2\pi inx} \quad \text{if } n < 0,
	\end{array} \right.\\
		&\quad\xi_k(u_{k,0}^{[j],-}(y)) = (-1)^j \biggl( j u_{2-k,0}^{[j-1],+}(y) + (1-k) u_{2-k,0}^{[j],+}(y) \biggr),\\
		&\quad\xi_k(u_{k,0}^{[j],+}(y)) = (-1)^{j-1} j u_{2-k,0}^{[j-1],-}(y),
\end{align*}
where we put $u_{k,n}^{[j],\pm}(y) = 0$ for any $j<0$.
\end{lmm}
For example, the function $u_{k,n}^{[j],-}(y) e^{2\pi inx}$ with $n \leq 0$ does not vanish by the action of $\xi_k \circ \Delta_k^j$. Since the termwise $\xi_k$-derivatives in the Fourier expansion in Proposition \ref{FWE} are guaranteed (see Remark after Proposition 2.3 in \cite{M}), we see that a function $f \in H_k^{r,!}$ with a positive integer $r$ is strictly in $H_k^{r-1/2,!}$ if and only if $c_{n,r-1}^- = 0$ for all $n \leq 0$ and $c_{n,r-1}^+ =0$ for all $n>0$. 
\begin{dfn}\label{defholpart}
	When $f(z) \in H_k^{r,!}$ has the Fourier expansion as in Proposition $\ref{FWE}$, we call the part
	\[
		\sum_{n>0} c_{n,0}^- u_{k,n}^{[0],-}(y)e^{2\pi inx} + c_{0,0}^+ + \sum_{n<0} c_{n,0}^+ u_{k,n}^{[0],+}(y)e^{2\pi inx}
	\]
	the holomorphic part of $f(z)$, which we denote by $f(z)^{\mathrm{hol}}$.
\end{dfn}

% -----------------------------------------------------------

\section{The Fourier expansion of the Maass-Poincar\'{e} series}\label{s3}%section3

From now, we explain analytic and algebraic properties of the Fourier coefficients of the Maass-Poincar\'{e} series based on \cite{DIT2, JKK}.

% -----------------------------------------------------------

\subsection{Analytic aspect}\label{s3.1}%section3.1

For $k \in \frac{1}{2}\mathbb{Z}$ and integer $m \in \mathbb{Z}$, let
\begin{align*}
	\varphi_{k,m}(z,s) :=  \left\{\begin{array}{ll}
		\Gamma(2s)^{-1}(4\pi |m|y)^{-k/2}M_{\mathrm{sgn}(m)\frac{k}{2}, s-1/2}(4\pi |m|y)e^{2\pi imx} \quad &\text{if } m \neq 0,\\
		y^{s-k/2} \quad &\text{if } m = 0,
	\end{array} \right.
\end{align*}
and define the corresponding Poincar\'{e} series by
\[
	\mathscr{P}_{k,m}(z,s) := \sum_{\gamma \in \Gamma_{\infty} \backslash \Gamma} (\varphi_{k,m}|_k\gamma) (z,s),
\]
where $\Gamma := \mathrm{SL}_2(\mathbb{Z})$ if $k \in \mathbb{Z}$ and $\Gamma := \Gamma_0(4)$ if $k \in \mathbb{Z} + 1/2$. This series is called the Maass-Poincar\'{e} series, and converges absolutely and uniformly on compact subsets in $\mathrm{Re}(s) > 1$. Moreover we set
\begin{align*}
	P_{k,m}(z,s) :=  \left\{\begin{array}{ll}
		\mathscr{P}_{k,m}(z,s) \quad &\text{if } k \in \mathbb{Z}, \\
		\mathrm{pr}_k^+(\mathscr{P}_{k,m}(z,s)) \quad &\text{if } k \in \mathbb{Z}+1/2,
	\end{array} \right.
\end{align*}
where $\mathrm{pr}_k^+$ is Kohnen's projection operator introduced by Kohnen \cite{Kohnen}. For $k = \lambda_k + 1/2$, this operator is given by
\[
	\mathrm{pr}_k^+(g) := (-1)^{\lfloor \frac{\lambda_k+1}{2} \rfloor}\frac{1}{2\sqrt{2}}\biggl(\sum_{\nu\ (\mathrm{mod}\ 4)}\bigl(g|_k A\bigr)|_k B_{\nu}\biggr) + \frac{1}{2}g
\]
where
\[
	A := \Biggl(\left[\begin{array}{cc}4 & 1 \\0 & 4 \end{array}\right], e^{\frac{\pi i}{4}} \biggr),\quad B_{\nu} := \left[\begin{array}{cc}1 & 0 \\4\nu & 1 \end{array}\right],
\]
and the slash operator for $(\gamma, \phi(z))$ is defined by $g|_k(\gamma, \phi(z)) := \phi(z)^{-2k}g(\gamma z)$. We can easily check that this $P_{k,m}(z,s)$ is an eigenfunction of the hyperbolic Laplacian,
\[
	\Delta_k P_{k,m}(z,s) = \bigl(s-\frac{k}{2}\bigr)\bigl(1-\frac{k}{2}-s\bigr)P_{k,m}(z,s).
\]
Furthermore, it can be meromorphically continued in $s$ to $\mathrm{Re}(s) > 1/2$ except for possibly finitely many simple poles at points of the discrete spectrum of $\Delta_k$. (See \cite[Section 3]{Fay}). It is known that the Fourier expansion of $P_{k,m}(z,s)$ can be expressed in terms of the Kloosterman sums and Bessel functions. We put
\begin{align}\label{mathcalW}
	\mathcal{W}_{k,n}(y,s) := \left\{\begin{array}{ll}
		\Gamma(s+\mathrm{sgn}(n)\frac{k}{2})^{-1}|n|^{k-1}(4\pi |n|y)^{-k/2}W_{\mathrm{sgn}(n)\frac{k}{2}, s-1/2}(4\pi|n|y) \quad &\text{if } n \neq 0,\\
		\dfrac{(4\pi)^{1-k}y^{1-s-k/2}}{(2s-1)\Gamma(s-k/2)\Gamma(s+k/2)} \quad &\text{if } n = 0.
	\end{array} \right.
\end{align}
Then, the following Fourier expansions are known.

\begin{prp}\cite[Theorem 3.2]{JKK}\label{JKKthm}
	If $k$ and $m$ are integers and $\mathrm{Re}(s)>1$, then the Poincar\'{e} series $P_{k,m}(z,s)$ has the Fourier expansion
\[
	P_{k,m}(z,s) = \varphi_{k,m}(z, s) + \sum_{n \in \mathbb{Z}} c_{k,m}(n,s)\mathcal{W}_{k,n}(y,s)e^{2\pi inx},
\]
where the coefficients $c_{k,m}(n,s)$ are given by
\begin{align*}
	2\pi i^{-k}\sum_{c>0} \frac{K(m,n,c)}{c} \times \left\{\begin{array}{ll}
		|mn|^{\frac{1-k}{2}}J_{2s-1}\bigl(\frac{4\pi \sqrt{|mn|}}{c}\bigr) &\text{if } mn>0,\\
		|mn|^{\frac{1-k}{2}}I_{2s-1}\bigl(\frac{4\pi \sqrt{|mn|}}{c}\bigr) &\text{if } mn<0,\\
		2^{k-1}\pi^{s+k/2-1}|m+n|^{s-k/2}c^{1-2s} &\text{if } mn=0, m+n \neq 0,\\
		2^{2k-2}\pi^{k-1}\Gamma(2s)(2c)^{1-2s} &\text{if } m=n=0.
	\end{array} \right.
\end{align*}
Here $I_s(y)$ and $J_s(y)$ are Bessel-functions, and we put
\[
	K(m,n,c) := \sum_{\substack{d (c)^*\\ad \equiv 1 (c)}}e\biggl(\frac{am+dn}{c}\biggr),\quad \text{with }e(x):=e^{2\pi ix}
\]
called the Kloosterman sum.
\end{prp}

\begin{prp}\cite[Theorem 4.4]{JKK}\label{Theorem 4.4}
	Let $k = \lambda_k + 1/2$. For any $m$ and $s$ satisfying $(-1)^{\lambda_k}m \equiv 0, 1\ (\mathrm{mod}\ 4)$ and $\mathrm{Re}(s)>1$, the function $P_{k,m}(z,s) = \mathrm{pr}_k^+(\mathscr{P}_{k,m}(z,s))$ has the Fourier expansion
\[
	P_{k,m}(z,s) = \varphi_{k,m}(z,s) + \sum_{(-1)^{\lambda_k}n \equiv 0,1 (4)}b_{k,m}(n,s)\mathcal{W}_{k,n}(y,s)e^{2\pi inx},
\]
where the coefficients $b_{k,m}(n,s)$ are given by
\begin{align*}
	2\pi i^{-k}\sum_{c>0}\biggl(1+\biggl(\frac{4}{c}\biggr)\biggr)\frac{\tilde{K}_k(m,n,4c)}{4c} \times \left\{\begin{array}{ll}
		|mn|^{\frac{1-k}{2}}J_{2s-1}\bigl(\frac{4\pi \sqrt{|mn|}}{4c}\bigr) &\text{if } mn>0,\\
		|mn|^{\frac{1-k}{2}}I_{2s-1}\bigl(\frac{4\pi \sqrt{|mn|}}{4c}\bigr) &\text{if } mn<0,\\
		2^{k-1}\pi^{s+k/2-1}|m+n|^{s-k/2}(4c)^{1-2s} &\text{if } mn=0, m+n \neq 0,\\
		2^{2k-2}\pi^{k-1}\Gamma(2s)(8c)^{1-2s} &\text{if } m=n=0.
	\end{array} \right.
\end{align*}
Here $\tilde{K}_k(m,n,c)$ is the generalized Kloosterman sum
\[
	\tilde{K}_k(m,n,c) := \sum_{\substack{d (c)^*\\ad \equiv 1 (c)}}\biggl(\frac{c}{d}\biggr)\epsilon_d^{2k} e\biggl(\frac{am+dn}{c}\biggr).
\]
\end{prp}

We note two symmetric properties for the coefficients $b_{k,m}(n,s)$. First, it follows immediately from the above explicit formula that $b_{k,m}(n,s) = b_{k,n}(m,s)$. Moreover the generalized Kloosterman sum satisfies
\begin{align*}
	\tilde{K}_{3/2}(m,n,c) &= -i\tilde{K}_{1/2}(-m,-n,c),\\
	\tilde{K}_{k+2}(m,n,c) &= \tilde{K}_k(m,n,c).
\end{align*}

\begin{lmm}\label{bkm}
	Let $k = \lambda_k + 1/2$. Then
\begin{align*}
	b_{k,m}(n,s) = (-1)^{\lfloor \frac{\lambda_k+1}{2} \rfloor}b_{1/2,(-1)^{\lambda_k}m}((-1)^{\lambda_k}n,s)  \times \left\{\begin{array}{ll}
		|mn|^{\frac{1-2k}{4}} &\text{if } m \neq 0, n \neq 0,\\
		2^{k-\frac{1}{2}}\pi^{\frac{2k-1}{4}}|m+n|^{-\frac{2k-1}{4}} &\text{if } mn=0, m+n \neq 0,\\
		2^{2k-1}\pi^{k-\frac{1}{2}} &\text{if } m=n=0.
	\end{array} \right.
\end{align*}
\end{lmm}

Consequently, our goal is shifted to investigate the coefficients $b_{1/2,m}(n,s)$. As explained in the next section, this coefficient was studied by Duke-Imamo\={g}lu-T\'{o}th \cite{DIT2}. 

% -----------------------------------------------------------

\subsection{Algebraic aspect}\label{s3.2}%section3.2

For an integer $m$, we consider the Niebur-Poincar\'{e} series \cite{N} $G_m(z,s)$ defined by
\[
	G_m(z,s) := \sum_{\gamma \in \mathrm{SL}_2(\mathbb{Z})_{\infty} \backslash \mathrm{SL}_2(\mathbb{Z})} (\phi_m|_0\gamma) (z,s),
\]
where $\phi_m(z,s)$ is defined by
\begin{align*}
	\phi_m(z,s) :=  \left\{\begin{array}{ll}
		2\pi|m|^{\frac{1}{2}}y^{\frac{1}{2}}I_{s-\frac{1}{2}}(2\pi|m|y)e^{2\pi imx} \quad &\text{if } m \neq 0,\\
		y^s \quad &\text{if } m = 0.
	\end{array} \right.
\end{align*}
Here there is one remark about the relation to the Maass-Poincar\'{e} series. Since it holds that \cite[(9.235)]{GR}
\[
	M_{0,\nu}(z) = 2^{2\nu}\Gamma(\nu+1)z^{1/2}I_{\nu}\bigl(\frac{z}{2}\bigr),
\]
for $m\neq 0$, we have
\[
	\varphi_{0,m}(z,s) = \Gamma(2s)^{-1}2^{2s-1}\Gamma(s+1/2)\pi^{-1/2}\phi_m(z,s).
\]
Moreover by the Legendre duplication formula
\[
	\Gamma(2s) = \frac{2^{2s-1}\Gamma(s)\Gamma(s+1/2)}{\pi^{1/2}},
\]
thus we have
\[
	\varphi_{0,m}(z,s) = \Gamma(s)^{-1}\phi_m(z,s) \quad \text{for }m\neq 0,
\]
that is, $P_{0,m}(z,s) = \Gamma(s)^{-1}G_m(z,s)$ holds.\\

From now, we consider the modified traces 
\begin{align}\label{modtr}
	\widetilde{\mathrm{Tr}}_{d,D}(G_m(z,s)) := \left\{\begin{array}{ll}
		\mathrm{Tr}_{d,D}(G_m(z,s)), \quad &\text{if }dD<0,\\
		B(s)^{-1}\mathrm{Tr}_{d,D}(G_m(z,s)), \quad &\text{if }dD>0, dD \neq \square,
	\end{array} \right.
\end{align}
for a fundamental discriminant $D$ or $D=1$. Here we put $B(s) := 2^s\Gamma(s/2)^2/(2\pi \Gamma(s))$. This trace function was considered in \cite[Proposition 5]{DIT2}. Suppose that $d$ and $D$ are not both negative. Then for $\mathrm{Re}(s)>1$, Duke-Imamo\={g}lu-T\'{o}th \cite[(5.2)]{DIT2} showed

\begin{align}\label{key}
	\begin{split}
		\widetilde{\mathrm{Tr}}_{d,D}(G_m(z,s)) = \left\{\begin{array}{ll}
			\sum_{0<n|m}\bigl(\frac{D}{n}\bigr)b_{\frac{1}{2},d}\bigl(\frac{m^2D}{n^2}, \frac{s}{2}+\frac{1}{4}\bigr), \quad &\text{if }m \neq 0,\\
			2^{s-1}\pi^{-\frac{s+1}{2}}|D|^{\frac{s}{2}}L_D(s)b_{\frac{1}{2},d}\bigl(0, \frac{s}{2}+\frac{1}{4}\bigr), \quad &\text{if }m =0,
		\end{array} \right.
	\end{split}
\end{align}
where $L_D(s) := \sum_{n=1}^{\infty}\bigl(\frac{D}{n}\bigr)n^{-s}$ is the Dirichlet $L$-function. Furthermore, by using the M\"{o}bius inversion, this can be written in the form
\begin{align}\label{key2}
	b_{1/2,d}\bigl(m^2D,\frac{s}{2}+\frac{1}{4}\bigr) = \left\{\begin{array}{ll}
		\sum_{0<n|m}\mu(n)\bigl(\frac{D}{n}\bigr)\widetilde{\mathrm{Tr}}_{d,D}(G_{m/n}(z,s)), \quad &\text{if }m\neq 0,\\
		2^{1-s}\pi^{\frac{s+1}{2}}|D|^{-\frac{s}{2}}L_D(s)^{-1}\widetilde{\mathrm{Tr}}_{d,D}(G_0(z,s)), \quad &\text{if }m=0.
	\end{array} \right.
\end{align}

As a remark, for $d=0$ and $m \neq 0$, we see that $b_{1/2,0}(m^2D,s) = b_{1/2,m^2D}(0,s)$. For the remaining case of $d=m=0$ is given by
\begin{align}\label{Ibu}
	b_{1/2,0}(0,s) = \pi^{1/2}2^{5/2-6s}\Gamma(2s)\frac{\zeta(4s-2)}{\zeta(4s-1)}.
\end{align}
(see \cite[(2.24)]{DIT2} and \cite[Proposition 2.3]{IS}). We also remark that these equations (\ref{key}) and (\ref{key2}) do not hold when $d<0, D<0$. On the other hand, Jeon-Kang-Kim \cite{JKK3} considered another modification $\mathrm{Tr}_{d,D}^*$ to hold these equations even for $d<0$ and $D<0$. In conclusion, under some assumptions on discriminants $d$ and $D$, the Fourier coefficients $b_{k,m}(n,s)$ are expressed in terms of the modified traces $\widetilde{\mathrm{Tr}}_{d,D}(G_0(z,s))$.

% -----------------------------------------------------------

\section{Polyharmonic Maass forms, Lagarias-Rhoades}\label{s4}%section4

In this section, we give a basis for the space $H_k^r$ as a half-integral weight analogue of Lagarias-Rhoades' work \cite{LR}. Let $k \in \mathbb{Z} + 1/2$. Since we always assume Kohnen's plus-condition, the behavior of $f \in H_k^r$ at the cusps $0$ and $1/2$ is determinde by that at $i\infty$. We recall that $F_{k,m,r}(z)$ and $G_{k,m,r}(z)$ are defined as
\begin{align}\label{FGdef}
	P_{k,m}(z,s) = \left\{\begin{array}{ll}
		\sum_{r \in \mathbb{Z}} F_{k,m,r}(z)\bigl(s+\frac{k}{2}-1\bigr)^r \quad &\text{if } k \leq 1/2, \\
\ \\
		\sum_{r \in \mathbb{Z}} G_{k,m,r}(z)\bigl(s-\frac{k}{2}\bigr)^r \quad &\text{if } k \geq 3/2,
	\end{array} \right.
\end{align}
after taking the analytic continuation. In addition, we recall that the coefficients $F_{k,m,r}(z)$ and $G_{k,m,r}(z)$ vanish if $r< 0$ except for $F_{1/2,n^2,-1}(z) \neq 0$ with $n \geq 0$ as explained in Section \ref{s1}. By the same argument as \cite[Section 5]{M}, we see that these coefficients satisfy the following recurrence relations. For $m\neq 0$,
\begin{align}\label{xirec1}
	\begin{split}
		\xi_k F_{k,m,r}(z) &= (4\pi |m|)^{1-k}\biggl\{(1-k)G_{2-k,-m,r}(z) + G_{2-k,-m,r-1}(z)\biggr\},\\
		\xi_k G_{k,m,r}(z) &=(4\pi |m|)^{1-k}F_{2-k,-m,r-1}(z),
	\end{split}
\end{align}
and for $m =0$,
\begin{align}\label{xirec2}
	\begin{split}
		\xi_k F_{k,0,r}(z) &= (1-k)G_{2-k,0,r}(z) + G_{2-k,0,r-1}(z),\\
		\xi_k G_{k,0,r}(z) &= F_{2-k,0,r-1}(z).
	\end{split}
\end{align}

Our main goal in this section is to show the following theorem.

\begin{thm}\label{Theorem 4.1}
	Let $r \geq 1$ be an integer, and $k =\lambda_k + 1/2$. Then
	\begin{enumerate}
		\item For $k \leq -1/2$, $H_k^{1/2} = \{0\}$ and $\{F_{k,0,0}(z), \dots, F_{k,0,r-1}(z)\}$ is a basis for $H_k^r = H_k^{r+1/2}$.
		\item For $k = 1/2$, $\{F_{1/2,0,-1}(z), \dots, F_{1/2,0,r-2}(z)\}$ is a basis for $H_{1/2}^r = H_{1/2}^{r-1/2}$.
		\item For $k = 3/2$, $H_{3/2}^{1/2} = \{0\}$ and $\{G_{3/2,0,0}(z), \dots, G_{3/2,0,r-1}\}$ is a basis for $H_{3/2}^r = H_{3/2}^{r+1/2}$.
		\item For $k \geq 5/2$, we have $H_k^r = H_k^{r-1/2} = E_k^r + S_k$, where $S_k$ consists of holomorphic cusp forms on $\Gamma_0(4)$ and $E_k^r$ is spanned by $\{ G_{k,0,0}(z), \dots, G_{k,0,r-1}(z)\}$.
	\end{enumerate}
\end{thm}

% -----------------------------------------------------------

\subsection{Weight $1/2$ case}

This section is based on the work of Duke-Imamo\={g}lu-T\'{o}th \cite{DIT2}. They showed
\[
	\lim_{s \to \frac{3}{4}}\biggl(P_{1/2,0}(z,s)-\frac{\frac{3}{4\pi}\theta(z)}{s-3/4}\biggr) = 3\widehat{\mathbf{Z}}_+(z) = F_{1/2,0,0}(z) \in H_{1/2}^{3/2},
\]
where $\theta(z) := \sum_{r\in\mathbb{Z}}q^{r^2}$ and $\widehat{\mathbf{Z}}_+(z)$ was given by \cite[Theorem 4]{DIT2}
\[
	\widehat{\mathbf{Z}}_+(z) = \sum_{d>0} \frac{1}{\sqrt{d}}\mathrm{Tr}_{d,1}(1)q^d + \text{(non-holomorphic part)},
\]
with suitable terms $\mathrm{Tr}_{d,1}(1)$ for square $d$. (Note that our definitions of $\mathrm{Tr}_{d,1}(1)$ and the holomorphic part are slightly different from theirs). This formula is an analogue of the Kronecker limit formula, and one example that the generating function of trace-values is realized as the holomorphic part of a polyharmonic weak Maass form. Consider the Laurent expansion
\[
	P_{1/2,0}(z,s) = \sum_{r = -1}^{\infty}F_{1/2,0,r}(z)(s-3/4)^r.
\]
Then by the important property
\[
	\Delta_{1/2} P_{1/2,0}(z,s) = \biggl(s-\frac{1}{4}\biggr)\biggl(\frac{3}{4}-s\biggr)P_{1/2,0}(z,s),
\]
we see that $\Delta_{1/2}^{r+2} F_{1/2,0,r}(z) = 0$. Furthermore from Proposition \ref{Theorem 4.4}, for $\mathrm{Re}(s)>1$ we have
\[
	P_{1/2,0}(z,s) = y^{s-1/4} + \sum_{n \equiv 0, 1 (4)} b_{1/2,0}(n,s) \mathcal{W}_{1/2,n}(y,s) e^{2\pi inx}.
\]
By the analytic continuation of $P_{1/2,0}(z,s)$ to $s = 3/4$, we can obtain the Fourier expansion of the function $F_{1/2,0,r}(z)$ from this Fourier expansion. (For more details, see \cite[Section 2]{DIT2}). This function $P_{1/2,0}(z,s)$ has no exponentially growing terms, thus we see that each coefficient $F_{1/2,0,r}(z)$ is a polyharmonic Maass form of weight $1/2$ and depth $r+2$ (actually depth $r+3/2$). Now we show the following lemma.

\begin{lmm}
The set $\{F_{1/2,0,-1}(z), \dots, F_{1/2,0,r-2}(z)\}$ is a basis for $H_{1/2}^r = H_{1/2}^{r-1/2}$.
\end{lmm}

\begin{proof}
This proof is based on \cite[Proof of Proposition 10.1]{LR}. It is known that the space $H_{1/2}^{1/2} = M_{1/2}^+$ has dimension $1$ and is spanned by the theta function $\theta(z)$. By Proposition \ref{FWE}, every $f \in H_{1/2}^1$ has a Fourier expansion of the form
\[
	f(z) = \sum_{n \equiv 0, 1 (4)}c_n^- u_{1/2,n}^{[0],-}(y)e^{2\pi inx} + c_0^+.
\]
By the action of $\xi_{1/2}$, we have
\begin{align*}
	\xi_{1/2}f(z) &= \frac{1}{2}c_0^- u_{3/2,0}^{[0],+}(y) - \sum_{0<n \equiv 0,3 (4)} c_{-n}^- u_{3/2,n}^{[0],-}(y)e^{2\pi inx}\\
	&= \frac{1}{2}c_0^-  - \sum_{0<n \equiv 0,3 (4)} c_{-n}^- (4\pi n)^{3/4}q^n,
\end{align*}
that is, $\xi_{1/2}f(z)$ is a holomorphic modular form of weight $3/2$. However it is well-known that the space $M_{3/2}$ without Kohnen's plus-condition is generated by $\theta(z)^3$. (see \cite[Proposition 4 in Section IV-1]{Kob}). Since the function $\theta(z)^3$ does not satisfy the plus-condition, the result $\xi_{1/2}f(z)$ should be $0$. Thus we obtain $f$ has depth $1/2$ actually. Similarly we can show $H_{1/2}^{r} = H_{1/2}^{r-1/2}$ by induction. Next we show $\mathrm{dim} H_{1/2}^{r+1} \leq \mathrm{dim} H_{1/2}^{r} + 1$. We assume that $\mathrm{dim}H_{1/2}^r =m$ and $\mathrm{dim}H_{1/2}^{r+1} \geq m+2$. Then we can take $m+2$ linearly independent functions $f_0(z) = \theta(z), f_1(z), \dots f_{m+1}(z) \in H_{1/2}^{r+1}$. On the other hand, for $1 \leq i \leq m+1$, it holds that $0 \neq \Delta_{1/2}f_{i}(z) \in H_{1/2}^r$. By our assumption of $\mathrm{dim}H_{1/2}^r =m$, there are some constants $\alpha_i \in \mathbb{C}$ such that $\sum_{i=1}^{m+1}\alpha_i \Delta_{1/2}f_i(z) = 0$. Thus $g(z) := \sum_{i=1}^{m+1} \alpha_i f_i(z)$ satisfies $\Delta_{1/2}g(z) =0$, that is, $g(z) = \theta(z)$ up to a constant multiple. This contradicts our assumption.
\end{proof}

% -----------------------------------------------------------

\subsection{Weight $3/2$ case}\label{s4.2}

This section is based on the work of Jeon-Kang-Kim \cite{JKK}. As we mentioned above, there is no holomorphic modular form satisfying the plus-condition of weight $3/2$ . They showed
\[
	P_{3/2,0}(z,\frac{3}{4}) = G_{3/2,0,0}(z) = -12 E_{3/2}(z),
\]
where $E_{3/2}(z) \in H_{3/2}^1$ is Zagier's Eisenstein series of weight $3/2$ given by
\[
	E_{3/2}(z) =\sum_{d \leq 0}H(|d|)q^{-d} + \frac{1}{16\pi \sqrt{y}}\sum_{n \in \mathbb{Z}}\beta(4\pi n^2y)q^{-n^2}.
\]
Here $\beta(s) = \int_1^{\infty}t^{-3/2}e^{-st}dt$. From Proposition \ref{Theorem 4.4}, we have
\[
	P_{3/2,0}(z,s) = y^{s-3/4} + \sum_{n \equiv 0,3 (4)}b_{3/2,0}(n,s)\mathcal{W}_{3/2,n}(y,s)e^{2\pi inx}.
\]
In the same way as the weight $1/2$ case, we can show the following lemma.
\begin{lmm}
	The set $\{G_{3/2,0,0}(z), \dots, G_{3/2,0,r-1}(z)\}$ is a basis for $H_{3/2}^r = H_{3/2}^{r+1/2}$.
\end{lmm}

% -----------------------------------------------------------

\subsection{Weight $k \leq -1/2$ or $5/2 \leq k$ cases}

For $k \geq 5/2$, the function
\begin{align*}
	P_{k,0}(z,s) &= \mathrm{pr}_k^+(\mathscr{P}_{k,0}(z,s))\\
	&= \mathrm{pr}_k^+\biggl(\sum_{\gamma \in \Gamma_0(4)_{\infty} \backslash \Gamma_0(4)} y^{s-k/2}\big|_k \gamma \biggr)
\end{align*}
at $s=k/2$ is known as Cohen's Eisenstein series. (see \cite[Proposition 6 in Section IV-2]{Kob}). Then we see that $H_k^{1/2}$ is spanned by $G_{k,0,0}(z) = P_{k,0}(z, k/2)$ and holomorphic cusp forms $S_k$. As for the case of $k \leq -1/2$, there is no holomorphic modular form of weight $k$. The remaining proof is the same as above. As for a cusp form $f \in S_k$, by \cite[Section 6.3]{LR} or \cite[Proposition 5.13]{BFOR}, a preimage $F$ such that $\xi_{2-k}F = f$ is not in $H_{2-k}^1$ but in $H_{2-k}^{1,!}$.

% -----------------------------------------------------------

\section{Duke-Jenkins basis}\label{s5}

Duke-Jenkins \cite{DJ2, DJ} constructed a standard basis for the space $H_k^{1/2,!} = M_k^!$ of weakly holomorphic modular forms. For $k = \lambda_k + 1/2$ with an integer $\lambda_k \in \mathbb{Z}$, we define an integer $\ell_k$ by $2\lambda_k = 12\ell_k + k'$ where $k' \in \{0, 4, 6, 8, 10, 14\}$. Moreover we put
\begin{align*}
	A_k := \left\{\begin{array}{ll}
		2\ell_k - (-1)^{\lambda_k} &\text{if } \ell_k \text{ is odd},\\
		2\ell_k &\text{if } \ell_k \text{ is even}.\\
	\end{array} \right.
\end{align*}
For each integer $m \geq -A_k$ satisfying $(-1)^{\lambda_k-1}m \equiv 0, 1 (4)$, there exists the unique weakly holomorphic modular form $f_{k,m}(z)$ with Fourier expansion of the form
\[
	f_{k,m}(z) = q^{-m} + \sum_{\substack{n > A_k\\ (-1)^{\lambda_k}n \equiv 0, 1 (4)}}a_k(m,n)q^n.
\]
In particular, these coefficients satisfy the duality relation
\[
	a_k(m,n) = -a_{2-k}(n,m).
\]
Then Duke-Jenkins showed that the set $\{ f_{k,m}(z)\ |\ m \geq -A_k, (-1)^{\lambda_k-1} m \equiv 0, 1 (4)\}$ is a basis for $H_k^{1/2,!}$.\\

On the other hand, the functions $F_{k,m,r}(z)$ and $G_{k,m,r}(z)$ with suitable $m$ and $r$ also span the space $H_k^{1/2,!}$. In particular for $k \geq 3/2$, we recall Petersson's following work.

\begin{thm}\cite[Satz 3, special case, modified version]{P}\label{Lemma 5.1}
	Let $k \in \frac{1}{2} \mathbb{Z}$ with $k \geq 3/2$ and $\mathcal{I}$ be a finite set of positive integers. Then
	\begin{align*}
		\sum_{m \in \mathcal{I}} \overline{\alpha_m} G_{k,m,0}(z) \equiv 0
	\end{align*}
	if and only if there exists a weakly holomorphic modular form of weight $2-k$ with principal part at $\infty$ equal to
	\begin{align*}
		\sum_{m \in \mathcal{I}}\frac{\alpha_m}{m^{k-1}}q^{-m}.
	\end{align*}
\end{thm}

\begin{proof}
	Here we prove the case of $k \in \mathbb{Z} + 1/2$. Similarly we can show it for $k \in \mathbb{Z}$. Since $G_{k,m,-1}(z) = 0$ for any $k \geq 3/2$ and $m >0$, by (\ref{xirec1}) we have
	\begin{align*}
		0 &= \sum_{m \in \mathcal{I}}\overline{\alpha_m}G_{k,m,0}(z) = \sum_{m \in \mathcal{I}}\xi_{2-k}\biggl(\alpha_m \frac{(4\pi m)^{1-k}}{k-1}F_{2-k,-m,0}(z)\biggr)\\
		&= \frac{(4\pi)^{1-k}}{k-1}\xi_{2-k} \bigg( \sum_{m \in \mathcal{I}} \frac{\alpha_m}{m^{k-1}} F_{2-k,-m,0}(z)\bigg).
	\end{align*}
	Thus the inner sum is a weakly holomorphic modular form of weight $2-k$ having the Fourier expansion of the form
	\[
		\sum_{m \in \mathcal{I}} \frac{\alpha_m}{m^{k-1}} q^{-m}.
	\]
	This concludes the proof.
\end{proof}

From now on, we reveal the relations between Duke-Jenkins' functions $f_{k,m}(z)$ and our functions $F_{k,m,r}(z)$ and $G_{k,m,r}(z)$. First, let $k = \lambda_k+1/2 \leq -1/2$ and $m>0$. Since $P_{k,m}(z,s)$ converges in $\mathrm{Re}(s) > 1$, we immediately see that
\begin{align*}
	F_{k,-m,0}(z) &= P_{k,-m}(z,1-k/2)\\
	&= \varphi_{k,-m}(z,1-k/2) + \sum_{(-1)^{\lambda_k}n \equiv 0,1 (4)}b_{k,-m}(n,1-k/2)\mathcal{W}_{k,n}(y,1-k/2)e^{2\pi inx}\\
	&= q^{-m}-\frac{\Gamma(1-k,4\pi my)}{\Gamma(1-k)}q^{-m}+ \sum_{(-1)^{\lambda_k}n \equiv 0,1 (4)}b_{k,-m}(n,1-k/2)\mathcal{W}_{k,n}(y,1-k/2)e^{2\pi inx}.
\end{align*}
Here we use the facts (\ref{useful1}), (\ref{useful2}), (\ref{MMW}), and $W_{\mu, -\nu}(y) = W_{\mu, \nu}(y)$. Comparing with the Duke-Jenkins basis $f_{k,m}(z) = q^{-m} + \sum_{\substack{n > A_k\\(-1)^{\lambda_k}n \equiv 0,1 (4)}}a_k(m,n)q^n$, for $k \leq -1/2$ and $m \geq -A_k >0$, we see that 
\[
f_{k,m}(z) - \biggl\{ F_{k,-m,0}(z) + \sum_{\substack{A_k < n < 0\\(-1)^{\lambda_k}n \equiv 0,1(4)}}a_k(m,n)F_{k,n,0}(z)\biggr\}
\]
is a harmonic function and bounded on the upper half plane $\mathfrak{H}$. Thus this difference is a constant, that is, equal to $0$. \\

Next we consider the case of $k = \lambda_k + 1/2 \geq 5/2$. By Theorem \ref{Lemma 5.1} and an easy remark
\begin{align*}
	A_k + A_{2-k} = \left\{\begin{array}{ll}
		-1 &\text{if } k = \lambda_k + 1/2 \text{ with } \lambda_k \equiv \ell_k (2),\\
		-3 &\text{if } k = \lambda_k + 1/2 \text{ with } \lambda_k \not\equiv \ell_k (2),\\
	\end{array} \right.
\end{align*}
we see that $\{G_{k,m,0}(z)\ |\ 0 < m \leq A_k\}$ is a basis for $S_k$. As for $k \geq 5/2$ and $-m < 0$,
\begin{align*}
	G_{k,-m,0}(z) &= P_{k,-m}(z,k/2)\\
	&= \varphi_{k,-m}(z,k/2) + \sum_{(-1)^{\lambda_k}n \equiv 0,1 (4)}b_{k,-m}(n,k/2)\mathcal{W}_{k,n}(y,k/2)e^{2\pi inx}\\
	&= \frac{1}{\Gamma(k)}\biggl\{q^{-m} + \sum_{(-1)^{\lambda_k}n \equiv 0,1 (4), n>0}n^{k-1}b_{k,-m}(n,k/2)q^n\biggr\}.
\end{align*}
Similarly we have that
\[
	f_{k,m}(z) - \Gamma(k)G_{k,-m,0}(z)
\]
is a holomorphic cusp form for weight $k \geq 5/2$. For $m=0$, we have
\begin{align*}
	G_{k,0,0}(z) &= P_{k,0}(z,k/2)\\
	&= 1 + \sum_{(-1)^{\lambda_k}n \equiv 0,1 (4), n>0}b_{k,0}(n,k/2)\frac{n^{k-1}}{\Gamma(k)}q^n,
\end{align*}
that is, $f_{k,0}(z) - G_{k,0,0}(z)$ is a holomorphic cusp form.\\

Finally, we consider the cases of $k = 1/2$ and $3/2$ separately. In these cases, the coefficients $b_{k,m}(n,s)$ has a possible pole at $s=3/4$. For $k=1/2$, Duke-Imamo\={g}lu-T\'{o}th showed

\begin{lmm}\cite[Section 2]{DIT2}\label{lemma5.2}
	Let $m \equiv 0,1(4)$.
	\begin{enumerate}
		\item If $m=0$, then $F_{1/2,0,-1}(z) = \frac{3}{4\pi}\theta(z) \in M_{1/2}$.
		\item If $m<0$, then $F_{1/2,m,0}(z) \in M_{1/2}^!$.
	\end{enumerate}
\end{lmm}

In the same way, for $m>0$ we see that
\[
	f_{1/2,m}(z) = F_{1/2,-m,0}(z),
\]
and $f_{1/2,0}(z) = \theta(z) = \frac{4\pi}{3}F_{1/2,0,-1}(z)$. As for $k = 3/2$, Jeon-Kang-Kim showed the following.

\begin{lmm}\cite[Proposition 5.1]{JKK}
	Let $m \equiv 0,3 (4)$.
	\begin{enumerate}
		\item If $m >0$, then $G_{3/2,m,0}(z) = 0$.
		\item If $m=0$, then $G_{3/2,0,0}(z) = -12E_{3/2}(z) \in H_{3/2}^1$.
		\item If $m<0$ and $-m$ is not a square, then $G_{3/2,m,0}(z) \in M_{3/2}^!$.
		\item If $m<0$ and $-m$ is a non-zero square, then
		\begin{align*}
			G_{3/2,m,0}(z) - \frac{4}{\sqrt{\pi}}G_{3/2,0,0}(z) \in M_{3/2}^!,
		\end{align*}
	\end{enumerate}
	where $E_{3/2}(z)$ is Zagier's Eisenstein series given in Section $\ref{s4.2}$.
\end{lmm}

More precisely, for $m<0$ but $-m \neq \square$, 
\[
	G_{3/2,m,0}(z) = \frac{2}{\sqrt{\pi}}\biggl(q^m + O(q) \biggr),
\]
and for $-m = \square$,
\[
	G_{3/2,m,0}(z) - \frac{4}{\sqrt{\pi}}G_{3/2,0,0}(z) = \frac{2}{\sqrt{\pi}}\biggl(q^m - 2 + O(q)\biggr).
\]

In conclusion, we obtain the following proposition.

\begin{prp}
	For $k = \lambda_k + 1/2$ with $\lambda_k \in \mathbb{Z}$, we define an integer $\ell_k$ by $2\lambda_k = 12\ell_k+k'$ where $k' \in \{0, 4, 6, 8, 10, 14\}$. Moreover we put 
\begin{align*}
	A_k := \left\{\begin{array}{ll}
		2\ell_k - (-1)^{\lambda_k} &\text{if } \ell_k \text{ is odd},\\
		2\ell_k &\text{if } \ell_k \text{ is even}.
	\end{array} \right.
\end{align*}
For each integer $m \geq -A_k$ with $(-1)^{\lambda_k-1}m \equiv 0,1 (4)$, the unique weakly holomorphic modular form $f_{k,m}(z) = q^{-m} + \sum_{\substack{n>A_k\\(-1)^{\lambda_k}n \equiv 0,1(4)}}a_k(m,n)q^n$ is expressed in terms of the functions $F_{k,m,r}(z)$, $G_{k,m,r}(z)$ as follows.
\begin{enumerate}
	\item For $k \leq -1/2$,
	\begin{align*}
		f_{k,m}(z) = F_{k,-m,0}(z) + \sum_{\substack{A_k < n < 0\\(-1)^{\lambda_k}n \equiv 0,1(4)}}a_k(m,n)F_{k,n,0}(z).
	\end{align*}
	\item For $k = 1/2$ and $m>0$, $f_{1/2,m}(z) = F_{1/2,-m,0}(z)$, and $f_{1/2,0}(z) = \frac{4\pi}{3}F_{1/2,0,-1}(z) = \theta(z)$.
	\item For $k=3/2$, 
	\begin{enumerate}
		\item If $m$ is not a square, then $f_{3/2,m}(z) = \dfrac{\sqrt{\pi}}{2}G_{3/2,-m,0}(z)$.
		\item If $m$ is a non-zero square, then
		\begin{align*}
			f_{3/2,m}(z) = \frac{\sqrt{\pi}}{2}\biggl(G_{3/2,-m,0}(z) - \frac{4}{\sqrt{\pi}}G_{3/2,0,0}(z)\biggr).
		\end{align*}
	\end{enumerate}
	\item For $k \geq 5/2$, the set $\{G_{k,m,0}(z)\ |\ 0 < m \leq A_k\}$ is a basis for the space $S_k$ of holomorphic cusp forms.
	\begin{enumerate}
		\item For $m>0$, $f_{k,m}(z) - \Gamma(k)G_{k,-m,0}(z)$ is a holomorphic cusp form.
		\item For $m=0$, $f_{k,0}(z) - G_{k,0,0}(z)$ is a holomorphic cusp form.
		\item For $m<0$, $f_{k,m}(z)$ is a holomorphic cusp form. 
	\end{enumerate}
\end{enumerate}
\end{prp}

\begin{rmk*}
	For $k \geq 5/2$, we can also express the Duke-Jenkins basis in terms of our functions $G_{k,m,r}(z)$ explicitly. For example, Jeon-Kang-Kim \cite{JKK2} gave such expression by using the Petersson inner product. 
\end{rmk*}

% -----------------------------------------------------------

\section{Proof of Theorem \ref{Main2}}\label{s6}

First, we consider the case of $k \leq -1/2$. By Proposition \ref{Theorem 4.4} and the relation \cite[9.233 (1)]{GR}
\begin{align}\label{MMW}
	M_{\mu,\nu}(y) = \frac{\Gamma(1+2\nu)}{\Gamma(\nu-\mu+\frac{1}{2})}e^{\pi i\mu}\mathcal{M}^+_{\mu,\nu}(y) + \frac{\Gamma(1+2\nu)}{\Gamma(\nu+\mu+\frac{1}{2})}e^{-\pi i (\nu-\mu+\frac{1}{2})}W_{\mu,\nu}(y), \quad \text{for } 2\nu \not\in \mathbb{Z}_{<0},
\end{align}
we have a Fourier expansion of the form
\begin{align*}
	F_{k,m,r}(z) &= \frac{1}{r!}\frac{\partial^r}{\partial s^r} P_{k,m}(z,s)\bigg|_{s=1-k/2}\\
	&= \sum_{j=0}^r c_{m,j}^+ u_{k,m}^{[j],+}(y) e^{2\pi imx} + \sum_{(-1)^{\lambda_k}n \equiv 0, 1(4)}\sum_{j=0}^r c_{n,j}^-u_{k,n}^{[j],-}(y)e^{2\pi inx},
\end{align*}
for $m \neq 0$. Since the exponentially growing terms come from $u_{k,m}^{[j],+}(y)$, we see that the set $\{F_{k,m,r-1}(z)\ |\ (-1)^{\lambda_k}m \equiv 0, 1 (4)\}$ spans $\mathcal{H}_k^{r,!} \oplus \mathcal{H}_k^{r-1/2,!}$ by Proposition \ref{FWE} and Theorem \ref{Theorem 4.1}. We recall that $H_k^{r,!} = H_k^{r-1/2,!} \oplus \mathcal{H}_k^{r,!}$. Here we put
\[
	\tilde{F}_{k,m,r-1}(z) := F_{k,m,r-1}(z) +\sum_{\substack{A_k < n < 0\\(-1)^{\lambda_k}n \equiv 0,1(4)}}a_k(-m,n)F_{k,n,r-1}(z).
\]
Then we have
\begin{align*}
	\xi_k \circ \Delta_k^{r-1} F_{k,m,r-1}(z) &= -(4\pi|m|)^{1-k}(k-1)^r G_{2-k,-m,0}(z), \quad \text{for } 0 \neq m>A_k\ (\Longleftrightarrow 0 \neq -m \leq A_{2-k}),\\
	\xi_k \circ \Delta_k^{r-1} F_{k,0,r-1}(z) &= -(k-1)^r G_{2-k,0,0}(z),\\
	\Delta_k^{r-1}\tilde{F}_{k,-m,r-1}(z) &= (k-1)^{r-1}\biggl(F_{k,-m,0}(z) +\sum_{\substack{A_k < n < 0\\(-1)^{\lambda_k}n \equiv 0,1(4)}}a_k(m,n)F_{k,n,0}(z)\biggr)\\
	&=(k-1)^{r-1}f_{k,m}(z), \quad \text{for }m \geq -A_k.
\end{align*}
From these results, we see that the functions $F_{k,m,r-1}(z)$ and $\tilde{F}_{k,m,r-1}(z)$ form bases for $\mathcal{H}_k^{r,!}$ and $\mathcal{H}_k^{r-1/2,!}$, respectively. As for $k \geq 5/2$ and $m > A_k$, similarly we put
\[
	\tilde{G}_{k,m,r}(z) := m^{k-1}G_{k,m,r}(z) - \sum_{\substack{0 < n \leq A_k\\(-1)^{\lambda_k}n \equiv 0,1(4)}}a_k(-n,m) n^{k-1}G_{k,n,r}(z),
\]
then we have
\begin{align*}
	\xi_k \circ \Delta_k^{r-1} \tilde{G}_{k,m,r}(z) &= (4\pi)^{1-k}(1-k)^{r-1}\biggl(F_{2-k,-m,0}(z) + \sum_{\substack{A_{2-k}<n<0\\(-1)^{\lambda_{2-k}}n \equiv 0,1(4)}}a_{2-k}(m,n)F_{2-k,n,0}(z)\biggr)\\
	&=(4\pi)^{1-k}(1-k)^{r-1}f_{2-k,m}(z).
\end{align*}

For the remaining  cases $k = 1/2$ and $3/2$, we see that
\begin{align*}
&\xi_{1/2}\circ\Delta_{1/2}^{r-1} \biggl(F_{1/2,m,r-1}(z) - 8\sqrt{m}\delta_{\square}(m)F_{1/2,0,r-1}(z)\biggr)\\
&\quad= -\biggl(-\frac{1}{2}\biggr)^{r}(4\pi m)^{1/2}\biggl(G_{3/2,-m,0}(z) - \frac{4}{\sqrt{\pi}}\delta_{\square}(m)G_{3/2,0,0}(z)\biggr).\\
&\Delta_{3/2}^{r-1}\biggl(G_{3/2,m,r-1}(z) - \frac{4}{\sqrt{\pi}}\delta_{\square}(-m)G_{3/2,0,r-1}(z)\biggr) \\
&\quad=\biggl(-\frac{1}{2}\biggr)^{r-1}\biggl(G_{3/2,m,0}(z)- \frac{4}{\sqrt{\pi}}\delta_{\square}(-m)G_{3/2,0,0}(z)\biggr),
\end{align*}
and so on. By a similar argument as above, we conclude this proof.

% -----------------------------------------------------------

\section{Proof of Theorem \ref{Main4}}\label{s7}

Suppose that $d$ and $D$ are not both negative and that $dD$ is not a square number. By (\ref{key}), we have
\begin{align*}
	\begin{split}
		\widetilde{\mathrm{Tr}}_{d,D}(G_m(z,s)) = \left\{\begin{array}{ll}
			\sum_{0<n|m}\bigl(\frac{D}{n}\bigr)b_{\frac{1}{2},d}\bigl(\frac{m^2D}{n^2}, \frac{s}{2}+\frac{1}{4}\bigr), \quad &\text{if }m \neq 0,\\
			2^{s-1}\pi^{-\frac{s+1}{2}}|D|^{\frac{s}{2}}L_D(s)b_{\frac{1}{2},d}\bigl(0, \frac{s}{2}+\frac{1}{4}\bigr), \quad &\text{if }m =0.
		\end{array} \right.
	\end{split}
\end{align*}
Since the sum runs over positive divisors $n$ of $m$, and the equation $\widetilde{\mathrm{Tr}}_{d,D}(G_{-m}(z,s)) = \widetilde{\mathrm{Tr}}_{d,D}(G_m(z,s))$ holds, it is enough to prove for a non-negative integer $m$.

% -----------------------------------------------------------

\subsection{The case: $d>0, D>0$, and $m > 0$}

We consider the linear combination of the Maass-Poincar\'{e} series
\[
	\mathcal{F}_{D,m}(z,s) := B(s)\Gamma\biggl(\frac{s+1}{2}\biggr)\sum_{n|m} \biggl(\frac{D}{n}\biggr)P_{\frac{1}{2}, \frac{m^2D}{n^2}}\biggl(z,\frac{s}{2}+\frac{1}{4}\biggr),
\]
where $B(s) = 2^s \Gamma(s/2)^2/(2\pi \Gamma(s))$. Then we compute the holomorphic part $\mathrm{LC}_{s=1}^r[\mathcal{F}_{D,m}(z,s)]^{\mathrm{hol}}$, where we define the $r$-th Laurent coefficient of the function $f(s)$ at $s=1$ by $\mathrm{LC}_{s=1}^r[f(s)]$. By Proposition \ref{Theorem 4.4} and the symmetric property $b_{k,m}(n,s) = b_{k,n}(m,s)$, we see the Fourier expansion of the form
\begin{align}
	 \mathcal{F}_{D,m}(z,s) &= B(s)\Gamma\biggl(\frac{s+1}{2}\biggr)\sum_{n|m} \biggl(\frac{D}{n}\biggr)\biggl[\varphi_{\frac{1}{2},\frac{m^2D}{n^2}}\biggl(z,\frac{s}{2}+\frac{1}{4}\biggr) \nonumber\\
	 &\hspace{100pt} + \sum_{d \equiv 0, 1(4)} b_{\frac{1}{2}, \frac{m^2D}{n^2}}\biggl(d, \frac{s}{2} + \frac{1}{4}\biggr)\mathcal{W}_{\frac{1}{2},d}\biggl(y, \frac{s}{2} + \frac{1}{4}\biggr)e^{2\pi idx}\biggr] \nonumber\\
	&= B(s)\Gamma\biggl(\frac{s+1}{2}\biggr)\sum_{n|m} \biggl(\frac{D}{m/n}\biggr)\biggl[\varphi_{\frac{1}{2},n^2D}\biggl(z,\frac{s}{2}+\frac{1}{4}\biggr) + b_{\frac{1}{2}, n^2D}\biggl(0, \frac{s}{2}+\frac{1}{4}\biggr)\mathcal{W}_{\frac{1}{2},0}\biggl(y, \frac{s}{2} + \frac{1}{4}\biggr)\biggr] \label{finpart}\\
	&\quad + \sum_{\substack{d \equiv 0, 1(4)\\ d \neq 0, dD = \square}} \sum_{n|m}\biggl(\frac{D}{m/n}\biggr)B(s)b_{\frac{1}{2}, n^2D}\biggl(d, \frac{s}{2} + \frac{1}{4}\biggr)\Gamma\biggl(\frac{s+1}{2}\biggr)\mathcal{W}_{\frac{1}{2},d}\biggl(y, \frac{s}{2} + \frac{1}{4}\biggr)e^{2\pi idx} \label{sqpart}\\
	&\quad + \sum_{\substack{d \equiv 0, 1(4)\\ dD \neq \square}} \mathrm{Tr}_{d,D}(G_m(z,s)) \Gamma\biggl(\frac{s+1}{2}\biggr)\mathcal{W}_{\frac{1}{2},d}\biggl(y, \frac{s}{2} + \frac{1}{4}\biggr)e^{2\pi idx}. \label{trpart}
\end{align}
Here we use the definition (\ref{modtr}) to get (\ref{trpart}).\\

First, we compute the finitely many additional part (\ref{finpart}). For the first term, by (\ref{MMW}) we have
\begin{align*}
	 B(s) &\Gamma \biggl(\frac{s+1}{2}\biggr) \varphi_{\frac{1}{2}, n^2D}\biggl(z, \frac{s}{2} + \frac{1}{4}\biggr)\\
	&= B(s)\frac{\Gamma((s+1)/2)}{\Gamma(s+1/2)}(4\pi n^2|D|y)^{-1/4}M_{\mathrm{sgn}(D)\frac{1}{4}, \frac{s}{2}-\frac{1}{4}}(4\pi n^2|D|y)e^{2\pi in^2 Dx}\\
	&=B(s)(4\pi n^2|D|y)^{-1/4}e^{2\pi in^2Dx}\\
	&\qquad \times \left\{\begin{array}{ll}
		\dfrac{\Gamma((s+1)/2)}{\Gamma(s/2)} e^{\frac{\pi i}{4}} \mathcal{M}_{\frac{1}{4}, \frac{s}{2}-\frac{1}{4}}^+(4\pi n^2Dy) + e^{-\frac{\pi is}{2}}W_{\frac{1}{4}, \frac{s}{2}-\frac{1}{4}}(4\pi n^2Dy), \quad &\text{if }D > 0,\\
		e^{-\frac{\pi i}{4}} \mathcal{M}_{-\frac{1}{4}, \frac{s}{2}-\frac{1}{4}}^+(4\pi n^2|D|y)+ \dfrac{\Gamma((s+1)/2)}{\Gamma(s/2)}e^{-\frac{\pi i(s+1)}{2}} W_{-\frac{1}{4}, \frac{s}{2}-\frac{1}{4}}(4\pi n^2|D|y), \quad &\text{if }D < 0.
	\end{array} \right.
\end{align*}
In addition by (\ref{useful1}), in our case of $D>0$ we obtain
\begin{align}\label{finhol}
	\mathrm{LC}_{s=1}^r \biggl[B(s)\Gamma\biggl(\frac{s+1}{2}\biggr) \varphi_{\frac{1}{2}, n^2D}\biggl(z, \frac{s}{2} + \frac{1}{4}\biggr)\biggr]^{\text{hol}} = \mathrm{LC}_{s=1}^r [B(s)e^{-\frac{\pi is}{2}}] q^{n^2D}.
\end{align}
Here we use the fact $W_{\mu, -\nu}(y) = W_{\mu, \nu}(y)$. For the second term, by the definition (\ref{mathcalW}), we have
\[
	\mathcal{W}_{\frac{1}{2},0}\biggl(y, \frac{s}{2}+\frac{1}{4}\biggr) = \frac{(4\pi)^{1/2}y^{1/2-s/2}}{(s-1/2)\Gamma(s/2)\Gamma(s/2+1/2)}.
\]
Then similarly we have 
\begin{align}\label{finhol2}
	\mathrm{LC}_{s=1}^r \biggl[B(s) \Gamma\biggl(\frac{s+1}{2}\biggr) b_{\frac{1}{2}, n^2D}\biggl(0, \frac{s}{2} + \frac{1}{4}\biggr) \mathcal{W}_{\frac{1}{2}, 0} \biggl(y, \frac{s}{2} + \frac{1}{4}\biggr)\biggr]^{\mathrm{hol}} = \mathrm{LC}_{s=1}^r \biggl[\frac{2^s \Gamma(s/2)}{\sqrt{\pi} \Gamma(s)}\frac{b_{\frac{1}{2}, n^2D}\bigl(0, \frac{s}{2} + \frac{1}{4}\bigr)}{s-1/2}\biggr].
\end{align}
Thus we see that
\begin{align}\label{addtr}
	\sum_{d \in \mathbb{Z}}\mathrm{Tr}_{d,D}^{\mathrm{add}}(F_{0,m,r})q^d := \mathrm{LC}_{s=1}^r[(\ref{finpart})]^{\mathrm{hol}} = \sum_{n|m} \biggl(\frac{D}{m/n}\biggr) [(\ref{finhol}) + (\ref{finhol2})].
\end{align}

We next consider the square-index part (\ref{sqpart}). By the definition of $\mathcal{W}_{1/2,n}(y,s)$ and (\ref{useful1}), we have
\begin{align}\label{sqtr}
	\begin{split}
		\sum_{\substack{0<d \equiv 0, 1(4)\\dD = \square}}d^{-1/2}\mathrm{Tr}_{d,D}^{\mathrm{sq}}(F_{0,m,r})q^d &:= \mathrm{LC}_{s=1}^r[(\ref{sqpart})]^{\mathrm{hol}}\\
		&= \sum_{\substack{0 < d \equiv 0, 1(4)\\ dD = \square}} \frac{1}{\sqrt{d}} \sum_{n|m}\biggl(\frac{D}{m/n}\biggr) \mathrm{LC}_{s=1}^r\biggl[B(s) b_{\frac{1}{2}, n^2D}\biggl(d, \frac{s}{2} + \frac{1}{4}\biggr)\biggr] q^d.
	\end{split}
\end{align}

Finally, the trace part (\ref{trpart}) is given by
\begin{align}\label{tr}
	\begin{split}
		\mathrm{LC}_{s=1}^r[(\ref{trpart})]^{\mathrm{hol}} &= \sum_{\substack{0 < d \equiv 0, 1(4)\\ dD \neq \square}} \frac{1}{\sqrt{d}} \mathrm{LC}_{s=1}^r \biggl[\mathrm{Tr}_{d,D}(G_m(z,s))\biggr] q^d\\
		&= \sum_{\substack{0 < d \equiv 0, 1(4)\\ dD \neq \square}} \frac{1}{\sqrt{d}} \mathrm{Tr}_{d,D}(F_{0,m,r}) q^d.
	\end{split}
\end{align}

Therefore, we obtain
\[
	\mathrm{LC}_{s=1}^r [\mathcal{F}_{D,m}(z,s)]^{\mathrm{hol}} = (\ref{addtr}) + (\ref{sqtr}) + (\ref{tr}).
\]

% -----------------------------------------------------------

\subsection{The case: $d>0, D<0$, and $m > 0$}

Similar as above, we consider the function
\[
	\mathcal{F}_{D,m}(z,s) := \Gamma\biggl(\frac{s+1}{2}\biggr)\sum_{n|m} \biggl(\frac{D}{n}\biggr)P_{\frac{1}{2}, \frac{m^2D}{n^2}}\biggl(z,\frac{s}{2}+\frac{1}{4}\biggr).
\]
By the facts (\ref{useful2}) and $\mathcal{M}_{\mu, \nu}^+(y) = \mathcal{M}_{\mu, -\nu}^+(y)$, we see that
\begin{align*}
	(4\pi n^2|D|y)^{-1/4} e^{-\frac{\pi i}{4}} \mathcal{M}_{-\frac{1}{4}, \frac{1}{4}}^+(4\pi n^2|D|y) e^{2\pi in^2 Dx} &= (4\pi n^2|D|)^{-1/4} e^{-\frac{\pi i}{4}} (4\pi n^2|D|e^{\pi i})^{1/4} q^{n^2D}\\
	&= q^{n^2D}.
\end{align*}
Then we see that
\begin{align*}
	\mathrm{LC}_{s=1}^r [\mathcal{F}_{D,m}(z,s)]^{\mathrm{hol}} &= \sum_{n|m} \biggl(\frac{D}{m/n}\biggr) \mathrm{LC}_{s=1}^r [1] q^{n^2D}\\
	&\quad + \sum_{n|m} \biggl(\frac{D}{m/n}\biggr) 2\sqrt{\pi} \cdot \mathrm{LC}_{s=1}^r \biggl[ \frac{b_{\frac{1}{2}, n^2D}\bigl(0, \frac{s}{2} + \frac{1}{4}\bigr)}{(s-1/2)\Gamma(s/2)}\biggr]\\
	&\quad + \sum_{0 < d \equiv 0, 1 (4)} \frac{1}{\sqrt{d}} \mathrm{Tr}_{d,D}(F_{0,m,r})q^d.
\end{align*}

Note that the square-index part does not appear in this case.

% -----------------------------------------------------------

\subsection{The case: $d>0, D>0$, and $m = 0$}

In this case, we take the function
\[
	\mathcal{F}_{D,0}(z,s) := 2^{s-1} \pi^{-\frac{s+1}{2}} |D|^{\frac{s}{2}} L_D(s) B(s) \Gamma\biggl(\frac{s+1}{2}\biggr) P_{\frac{1}{2}, 0}\biggl(z, \frac{s}{2} + \frac{1}{4}\biggr).
\]
By Proposition \ref{Theorem 4.4} and (\ref{Ibu}), we have the Fourier expansion
\begin{align*}
	P_{\frac{1}{2}, 0}\biggl(z, \frac{s}{2} + \frac{1}{4}\biggr) &= y^{s/2} + \frac{\zeta(2s-1)}{\zeta(2s)}\frac{2^{2-3s} \pi \Gamma(s+1/2) y^{1/2-s/2}}{(s-1/2)\Gamma(s/2)\Gamma(s/2+1/2)}\\
	& \quad + \sum_{\substack{d \equiv 0, 1(4) \\ d \neq 0}}b_{\frac{1}{2}, 0} \biggl(d,\frac{s}{2} + \frac{1}{4}\biggr) \mathcal{W}_{\frac{1}{2}, d}\biggl(y, \frac{s}{2} + \frac{1}{4}\biggr)e^{2\pi idx}.
\end{align*}

Thus we have
\begin{align*}
	\mathrm{LC}_{s=1}^r[\mathcal{F}_{D,0}(z,s)]^{\mathrm{hol}} &= \mathrm{LC}_{s=1}^r \biggl[\frac{|D|^{\frac{s}{2}} L_D(s)}{2^s \pi^{\frac{s+1}{2}}}\frac{\Gamma(s/2)\Gamma(s-1/2)}{\Gamma(s)} \frac{\zeta(2s-1)}{\zeta(2s)} \biggr]\\
	&\quad + \sum_{\substack{0 < d \equiv 0, 1 (4) \\ dD = \square}} \frac{1}{\sqrt{d}} \mathrm{LC}_{s=1}^r \biggl[2^{s-1} \pi^{-\frac{s+1}{2}} |D|^{\frac{s}{2}} L_D(s) B(s) b_{\frac{1}{2}, 0}\biggl(d, \frac{s}{2} + \frac{1}{4}\biggr)\biggr] q^d\\
	&\quad + \sum_{\substack{0 < d \equiv 0, 1 (4) \\ dD \neq \square}} \frac{1}{\sqrt{d}} \mathrm{Tr}_{d,D}(F_{0,0,r}) q^d.
\end{align*}

% -----------------------------------------------------------

\subsection{The case: $d>0, D<0$, and $m = 0$}

Similarly we take
\[
	\mathcal{F}_{D,0}(z,s) := 2^{s-1} \pi^{-\frac{s+1}{2}} |D|^{\frac{s}{2}} L_D(s) \Gamma\biggl(\frac{s+1}{2}\biggr) P_{\frac{1}{2}, 0}\biggl(z, \frac{s}{2} + \frac{1}{4}\biggr).
\]
Then we have
\begin{align*}
	\mathrm{LC}_{s=1}^r[\mathcal{F}_{D,0}(z,s)]^{\mathrm{hol}} &= \mathrm{LC}_{s=1}^r \biggl[ \frac{|D|^{\frac{s}{2}} L_D(s)}{2^{2s-1} \pi^{\frac{s-1}{2}}} \frac{\Gamma(s-1/2)}{\Gamma(s/2)} \frac{\zeta(2s-1)}{\zeta(2s)}\biggr] + \sum_{0 < d \equiv 0, 1 (4)} \frac{1}{\sqrt{d}} \mathrm{Tr}_{d,D}(F_{0,0,r}) q^d.
\end{align*}

% -----------------------------------------------------------

\subsection{The case: $d<0, D>0$, and $m > 0$}

We consider the linear combination
\[
	\mathcal{G}_{D,m}(z,s) := -\Gamma\biggl(\frac{s}{2}+1\biggr)\sum_{n|m}\biggl(\frac{D}{n}\biggr)\bigg|\frac{m}{n}\bigg|\sqrt{D}P_{\frac{3}{2}, -\frac{m^2D}{n^2}}\biggl(z,\frac{s}{2} + \frac{1}{4}\biggr).
\]
By Lemma \ref{bkm}, we have
\begin{align*}
	P_{\frac{3}{2}, -\frac{m^2 D}{n^2}} \biggl(z, \frac{s}{2} + \frac{1}{4}\biggr) &= \varphi_{\frac{3}{2}, -\frac{m^2D}{n^2}}\biggl(z, \frac{s}{2} + \frac{1}{4}\biggr) + \frac{ b_{\frac{3}{2}, -\frac{m^2D}{n^2}}\bigl(0, \frac{s}{2} + \frac{1}{4}\bigr)(4\pi)^{-\frac{1}{2}}y^{-s/2}}{(s-1/2)\Gamma(s/2-1/2)\Gamma(s/2+1)}\\
	&\quad - \sum_{\substack{d \equiv 0 ,3 (4) \\ d \neq 0}} \bigg| \frac{m^2}{n^2} dD\bigg|^{-\frac{1}{2}} b_{\frac{1}{2}, \frac{m^2D}{n^2}}\biggl(-d, \frac{s}{2} + \frac{1}{4}\biggr) \mathcal{W}_{\frac{3}{2}, d} \biggl(z, \frac{s}{2} + \frac{1}{4}\biggr) e^{2\pi i dx}.
\end{align*}
Here by (\ref{MMW}) again, we have
\begin{align*}
	\Gamma&\biggl(\frac{s}{2} + 1\biggr)\varphi_{\frac{3}{2}, -n^2D}\biggl(z, \frac{s}{2} + \frac{1}{4}\biggr)\\
	&= \frac{\Gamma(s/2+1)}{\Gamma(s+1/2)} (4\pi n^2 D y)^{-3/4} M_{-\frac{3}{4}, \frac{s}{2}-\frac{1}{4}}(4\pi n^2Dy) e^{-2\pi i n^2Dx}\\
	&= (4\pi n^2 D y)^{-3/4}  \biggl[e^{-\frac{3\pi i}{4}} \mathcal{M}_{-\frac{3}{4}, \frac{s}{2} - \frac{1}{4}}^+ (4\pi n^2Dy) + \dfrac{\Gamma(s/2+1)}{\Gamma(s/2-1/2)} e^{-\pi i(s/2+1)} W_{-\frac{3}{4}, \frac{s}{2} - \frac{1}{4}}(4\pi n^2Dy)\biggr] e^{-2\pi in^2Dx}.
\end{align*}
Thus we obtain
\begin{align*}
	\mathrm{LC}_{s=1}^r [\mathcal{G}_{D,m}(z,s)]^{\mathrm{hol}} &= -\sum_{n|m} \biggl(\frac{D}{m/n}\biggr) n \sqrt{D} \mathrm{LC}_{s=1}^r[1] q^{-n^2D} + \sum_{0 > d \equiv 0, 1 (4)} \mathrm{Tr}_{d, D}(F_{0,m,r}) q^{-d}.
\end{align*}

% -----------------------------------------------------------

\subsection{The case: $d<0, D>0$, and $m = 0$}

Finally, we consider the function
\[
	\mathcal{G}_{D,0}(z,s) := -2^{s-2} \pi^{-\frac{s}{2}-1} |D|^{\frac{s}{2}} L_D(s) \Gamma\biggl(\frac{s}{2}+1\biggr) P_{\frac{3}{2}, 0} \biggl(z, \frac{s}{2} + \frac{1}{4} \biggr).
\]
By Proposition \ref{Theorem 4.4}, Lemma \ref{bkm}, and (\ref{Ibu}), we have
\begin{align*}
	P_{\frac{3}{2}, 0}\biggl(z, \frac{s}{2} + \frac{1}{4}\biggr) &= y^{\frac{s-1}{2}} - \pi^{\frac{3}{2}} 2^{3-3s}  \frac{\zeta(2s-1)}{\zeta(2s)} \frac{\Gamma(s-1/2)(4\pi)^{-1/2}y^{-s/2}}{\Gamma(s/2-1/2)\Gamma(s/2+1)}\\
	&\quad - \sum_{\substack{d \equiv 0, 3 (4) \\ d \neq 0}} \frac{2\sqrt{\pi}}{\sqrt{|d|}}b_{\frac{1}{2}, 0}\biggl(-d, \frac{s}{2} + \frac{1}{4}\biggr) \mathcal{W}_{\frac{3}{2}, d}\biggl(y, \frac{s}{2} + \frac{1}{4} \biggr) e^{2\pi idx}.
\end{align*}
Therefore, we see that
\begin{align*}
	\mathrm{LC}_{s=1}^r [\mathcal{G}_{D,0}(z,s)]^{\mathrm{hol}} &= -\mathrm{LC}_{s=1}^r\bigg[2^{s-2} \pi^{-\frac{s}{2}-1} |D|^{\frac{s}{2}} L_D(s) \Gamma\biggl(\frac{s}{2}+1\biggr)\bigg] + \sum_{0 > d \equiv 0, 1 (4)} \mathrm{Tr}_{d,D}(F_{0,0,r}) q^{-d}.
\end{align*}

These conclude the proof of Theorem \ref{Main4}.\\

As an application, for a fundamental discriminant $D \neq 1$, we immediately see that
\begin{align*}
	\mathrm{LC}_{s=1}^{-1} [\mathcal{F}_{D,0}(z,s)]^{\text{hol}} &= \frac{\sqrt{|D|} L_D(1)}{\pi} \cdot 2F_{1/2,0,-1}(z) = \frac{3 \sqrt{|D|}L_D(1)}{2\pi^2} + \sum_{\substack{0 < d \equiv 0,1 (4) \\ dD \neq \square}} \frac{3}{\pi \sqrt{d}} \mathrm{Tr}_{d,D}(1) q^d,\\
	\mathrm{LC}_{s=1}^{-1} [\mathcal{G}_{D,0}(z,s)]^{\text{hol}} &= 0 =\frac{3}{\pi} \sum_{0 > d \equiv 0, 1(4)} \mathrm{Tr}_{d,D}(1) q^{-d}.
\end{align*}
Here we recall that $F_{0,0,-1}(z) = 3/\pi$. On the other hand, by Lemma \ref{lemma5.2} we have $F_{1/2,0,-1}(z) = (3/4\pi)\theta(z)$. Therefore for $D \neq 1$ we have
\begin{align}\label{twtrace}
	\mathrm{Tr}_{d,D}(1) = \left\{\begin{array}{ll}
		\dfrac{\sqrt{|dD|} L_D(1)}{\pi}, \quad &\text{if } d = \square,\\
		0, \quad &\text{if } d \neq \square.
	\end{array} \right.
\end{align}

% -----------------------------------------------------------

\section{Proof of Corollary \ref{MainCor}}\label{s8}

Finally, as an example, we compute the Fourier coefficients of the holomorphic part of $F_{1/2,0,0}(z)$. Throughout this section, we assume that a positive integer $d$ and the product $dD$ are not square numbers. By Theorem \ref{Main2}, the function $F_{1/2,0,0}(z)$ is a polyharmonic Maass form of weight $1/2$ and depth $3/2$. By the definition we have 
\begin{align*}
	F_{1/2,0,0}(z) &= \mathrm{LC}_{s=3/4}^0 [P_{1/2,0}(z,s)],
\end{align*}
where we also define the $r$-th Laurent coefficient of the function $f(s)$ at $s=3/4$ by $\mathrm{LC}_{s=3/4}^r[f(s)]$. By Proposition \ref{Theorem 4.4}, (\ref{useful1}), and (\ref{key2}), the $d$-th Fourier coefficient of its holomorphic part is given by
\begin{align}\label{LR}
	\mathrm{LC}_{s=3/4}^0 \bigg[\frac{b_{1/2,0}(d,s)}{d^{1/2} \Gamma(s+1/4)}\bigg] &= \mathrm{LC}_{s=3/4}^0 \bigg[ \frac{2^{3/2-2s} \pi^{s+1/4} |D|^{-s+1/4}}{d^{1/2} \Gamma(s+1/4)} \frac{\widetilde{\mathrm{Tr}}_{d,D}(G_0(z,2s-1/2))}{L_D(2s-1/2)} \bigg].
\end{align}
Here it is known that
\begin{align*}
	L_D(2s-1/2)^{-1} &= \left\{\begin{array}{ll}
		2(s-3/4) + O((s-3/4)^2) &\text{if } D=1,\\
		L_D(1)^{-1} + O(s-3/4) &\text{if } D \neq 1,\\
	\end{array} \right.\\
	G_0(z,2s-1/2) &= \frac{3}{2\pi} \frac{1}{s-3/4} -\frac{3}{\pi} \log(y|\eta(z)|^4) + C + O(s-3/4),
\end{align*}
where $C = (6/\pi)(\gamma - \log 2 - 6\zeta'(2)/\pi^2)$. The key point of this proof is that the left-hand side of (\ref{LR}) does not depend on $D$. We recall that $\mathrm{Tr}_{d,D}(1) = 0$ holds for $D \neq 1$ by (\ref{twtrace}). From these properties, we immediately see that if $D=1$,
\[
	\mathrm{LC}_{s=3/4}^0 \bigg[\frac{b_{1/2,0}(d,s)}{d^{1/2} \Gamma(s+1/4)}\bigg] = \frac{\pi}{\sqrt{d}} \frac{2}{B(1)} \mathrm{Tr}_{d,1}(3/2\pi) = \frac{3}{\sqrt{d}} \mathrm{Tr}_{d,1}(1),
\]
while for $D < 0$ we have
\[
	\mathrm{LC}_{s=3/4}^0 \bigg[\frac{b_{1/2,0}(d,s)}{d^{1/2} \Gamma(s+1/4)}\bigg] = \frac{3}{\sqrt{|dD|} L_D(1)} \mathrm{Tr}_{d,D}(-\log(y|\eta(z)|^4)),
\]
and for $D>1$ we have
\begin{align*}
	\mathrm{LC}_{s=3/4}^0 \bigg[\frac{b_{1/2,0}(d,s)}{d^{1/2} \Gamma(s+1/4)}\bigg] &= \frac{\pi}{\sqrt{|dD|}} \frac{1}{B(1)} \frac{\mathrm{Tr}_{d,D}((-3/\pi)\log(y|\eta(z)|^4))}{L_D(1)}\\
	&= \frac{3}{\sqrt{dD} L_D(1)} \mathrm{Tr}_{d,D}(-\log(y|\eta(z)|^4)).
\end{align*}
Comparing these right-hand sides, we have Corollary \ref{MainCor}.

% -----------------------------------------------------------

\end{document}